\documentclass[11pt]{article}%
\usepackage{amsfonts}
\usepackage{amsmath}
\usepackage{color}
\usepackage{geometry}
\usepackage{amssymb}
\usepackage{graphicx}%
\providecommand{\U}[1]{\protect\rule{.1in}{.1in}}
\newtheorem{theorem}{Theorem}

\newtheorem{corollary}[theorem]{Corollary}

\newtheorem{definition}[theorem]{Definition}

\newtheorem{lemma}[theorem]{Lemma}

\newtheorem{proposition}[theorem]{Proposition}
\newtheorem{remark}[theorem]{Remark}

\newenvironment{proof}[1][Proof]{\noindent\textbf{#1.} }{\ \rule{0.5em}{0.5em}}
\begin{document}

\title{Mean field limit of interacting filaments and vector valued non linear PDEs}
\author{Hakima Bessaih\thanks{University of Wyoming, Department of Mathematics, Dept.
3036, 1000 East University Avenue, Laramie WY 82071, United States,
bessaih@uwyo.edu}, Michele Coghi\thanks{ Scuola Normale Superiore, Pisa, P.zza dei Cavalieri 7, 56127, Pisa,
michele.coghi@gmail.com}, Franco Flandoli\thanks{ Dipartimento di Matematica, Universit\`{a} di Pisa, Largo Bruno Pontecorvo 5, 56127, Pisa, Italy, flandoli@dma.unipi.it}}
\date{}
\maketitle

\begin{abstract}
Families of $N$ interacting curves are considered, with long range, mean
field type, interaction. A family of curves defines a 1-current, concentrated on the
curves, analog of the empirical measure of interacting point particles.
This current is proved to converge, as $N$ goes to infinity, to a mean field current, solution of a nonlinear,
vector valued, partial differential equation. In the limit, each curve
interacts with the mean field current and two different curves have an
independence property if they are independent at time zero. This set-up
is inspired from vortex filaments in turbulent fluids, although for
technical reasons we have to restrict to smooth interaction, instead of
the singular Biot-Savart kernel. All these results are based on a
careful analysis of a nonlinear flow equation for 1-currents, its
relation with the vector valued PDE and the continuous dependence on the
initial conditions.
\end{abstract}

\section{Introduction}

Classical mean field theory deals with pointwise particles in $\mathbb{R}^{d}%
$, described by their position $X_{t}^{i,N}$, that satisfy dynamics of the
form
\begin{equation}
\frac{dX_{t}^{i,N}}{dt}=\frac{1}{N}\sum_{j=1}^{N}k\left(  X_{t}^{i,N}%
-X_{t}^{j,N}\right)  \label{interacting particles}%
\end{equation}
governed by the interaction kernel $k:$ $\mathbb{R}^{d}\rightarrow
\mathbb{R}^{d}$ (often a stochastic analog is considered but here we deal with
the deterministic case). Denoting by $S_{t}^{N}=\frac{1}{N}\sum_{i=1}%
^{N}\delta_{X_{t}^{i,N}}$ the empirical measure, if $k$ is bounded Lipschitz
continuous and $S_{0}^{N}$ weakly converges to a probability measure $\mu_{0}%
$, one can prove that $S_{t}^{N}$ weakly converges to a measure-valued
solution $\mu_{t}$ of the mean field equation%
\[
\frac{\partial\mu_{t}}{\partial t}+\operatorname{div}\left(  \left(  k\ast
\mu_{t}\right)  \mu_{t}\right)  =0
\]
with initial condition $\mu_{0}$, where $k\ast\mu_{t}$ is the vector field in
$\mathbb{R}^{d}$ with $i$-component given by the convolution $k_{i}\ast\mu
_{t}$; see \cite{Dob}.

Our aim is to develop an analogous result in the case when interacting points
are replaced by interacting curves, that we call "filaments" by inspiration
from the theory of vortex filaments in 3D fluids. The limit nonlinear PDE is
vector valued or, more precisely, \textit{current}-valued, as explained below.

The filament structures are curves in $\mathbb{R}^{d}$, $\gamma_{t}%
^{i,N}\left(  \sigma\right)  $, $i=1,...,N$, parametrized by $\sigma\in\left[
0,1\right]  $, and their interaction is described by the differential equation%
\[
\frac{\partial}{\partial t}\gamma_{t}^{i,N}\left(  \sigma\right)  =\sum
_{j=1}^{N}\alpha_{j}^{N}\int_{0}^{1}K\left(  \gamma_{t}^{i,N}\left(
\sigma\right)  -\gamma_{t}^{j,N}\left(  \sigma^{\prime}\right)  \right)
\frac{\partial}{\partial\sigma^{\prime}}\gamma_{t}^{j,N}\left(  \sigma
^{\prime}\right)  d\sigma^{\prime}%
\]
where $\alpha_{j}^{N}$ play the role of the factors $\frac{1}{N}$ in
(\ref{interacting particles}) and where now $K:\mathbb{R}^{d}\rightarrow
\mathbb{R}^{d\times d}$ is a smooth matrix-valued function (precisely, we need
$K$ of class $\mathcal{U}C_{b}^{3}(\mathbb{R}^{d},\mathbb{R}^{m})$, see
Section \ref{section preliminaries} for the definition). To the family of curves we associate a vector valued distribution (a
"current") $\xi_{t}:C_{b}\left(  \mathbb{R}^{d};\mathbb{R}^{d}\right)
\rightarrow\mathbb{R}$ formally given by %
\[
\xi_{t}^{N}=\sum_{j=1}^{N}\alpha_{j}^{N}\int_{0}^{1}\delta_{\gamma_{t}%
^{j,N}\left(  \sigma\right)  }\frac{\partial}{\partial\sigma}\gamma_{t}%
^{j,N}\left(  \sigma\right)  d\sigma
\]
which plays the role of the empirical measure $S_{t}^{N}$. The mean field
result will be that, under suitable assumptions on the initial conditions,
$\xi_{t}^{N}$ converges weakly to a current-valued solution $\xi_{t}$ of the vector-valued equation%
\begin{equation}
\frac{\partial\xi_{t}}{\partial t}+\operatorname{div}\left(  \left(  K\ast
\xi_{t}\right)  \xi_{t}\right)  =\xi_{t}\cdot\nabla\left(  K\ast\xi
_{t}\right)  \label{PDE for currents}%
\end{equation}
where $K\ast\xi_{t}$ is the vector field in $\mathbb{R}^{d}$ defined by
(\ref{def convol}) below and the meaning of the equation is given by
Definition \ref{solution}. Moreover, in the limit, each filament is coupled
only with the \textit{mean field} $\xi_{t}$:%
\begin{equation}
\frac{\partial}{\partial t}\gamma_{t}^{i}\left(  \sigma\right)  =\left(
K\ast\xi_{t}\right)  \left(  \gamma_{t}^{i}\left(  \sigma\right)  \right)
\label{filament and mean field}%
\end{equation}
and any two filaments in the limit have a suitable independence property, if the initial conditions are also independent (all these limit results require precise
statements, given in section \ref{subsection mean field}).

The investigation made here of interacting curves and the associated mean field PDE is motivated by the theory of vortex filaments in turbulent fluids. Starting from the simulations of \cite{VincMene}, a new vision of a three dimensional turbulent fluid appeared as a system composed of a large number of lower dimensional structures, in particular thin vortex structures. The idea is well described for instance by A. Chorin in his book \cite{Chorin}. For the purpose of turbulence, the investigation of large families of filaments was related to statistical properties, as we shall recall below. But, in parallel to statistical investigations, one of the natural questions is the relation between these families of filaments and the equations of fluid dynamics, the Euler or Navier-Stokes equations. In dimension 2, it is known that a proper mean field limit of point vortices leads to the 2D Euler equation. In dimension 3 this is an open problem, see for instance \cite{LionsMajda}. Our mean field result here is a contribution in this direction. We do not solve the true fluid dynamic problem, since we cannot consider Biot-Savart kernel $K$ yet, but at least for relatively smooth kernels we show that the expected result holds true.

Having mentioned the link with fluid dynamics and works on vortex filaments, let us give more details and some references. As we have already said, the importance of thin vortex structures in 3D turbulence has been discussed
intensively, especially after the striking simulations of \cite{VincMene}.
While the situation in the two-dimensional case is pretty understood, this is
not the case in the three-dimensional case. Chorin \cite{Chorin} has
emphasized both the similarities and differences between statistical theories
for heuristic models for ensembles of three-dimensional vortex filaments and
the earlier two-dimensional statistical theories for point vortices. Some
probabilistic models of vortex filaments based on the paths of stochastic
processes have been proposed in \cite{Gallavotti}, \cite{LionsMajda},
\cite{Flandoli}, \cite{FlaGub-2002}, \cite{FlaMin}, \cite{NuaRovTin}. The
importance of these models for the statistics of turbulence or for the
understanding of 3D Euler equations is of high importance. Let us mention that
the existence and uniqueness of solutions for the dynamics of vortex filaments
has been investigated in \cite{BerBes} and for a random vortex filament
\cite{BesGubRus}, \cite{BrzGubNek} and in \cite{ChanBes} in the case of
fractional Brownian motion. Of course, all the previous references mentioned
deals with a smoothened version of the dynamics which is related to a
mollified version of the Biot-Savart formula.

Statistical ensembles of vortex filaments arise many questions. One of them,
approached with success by Onsager and subsequent authors in dimension 2, is
the mean field limit of a dense collection of many interacting vortices. In
dimension 3 this question has been investigated successfully by P. L. Lions
and A. Majda. In \cite{LionsMajda}, they develop the first mathematically
rigorous equilibrium statistical theory for three-dimensional vortex filaments
in the context of a model involving simplified asymptotic equations for nearly
parallel vortex filaments. Their equilibrium Gibbs ensemble is written down
exactly through function space integrals; then a suitably scaled mean field
statistical theory is developed in the limit of infinitely many interacting
filaments. The mean field equations involved a novel Hartree-like problem. A
similar approach has been used for stochastic vortex filaments in
\cite{BesFla-2003}, \cite{BesFla-2004} where the Gibbs measure was based on a
previous rigorous definition introduced in \cite{FlaGub-2002}. The mean field
was proved to be solution of a variational formulation but given in an
implicit form.

As far as the content of the paper, section 2 is devoted to the  introduction of the space of 
currents (1-forms) provided with its strong and weak topologies. 
The push forward of 1-currents is defined with some properties.  In section 3, Lagrangian current dynamics are introduced. 
A flow equation for the current is defined 
by taking the push forward of an initial current under the flow of diffeomorphisms generated by a  general differential equation.  
The existence and uniqueness of maximal solutions for the flow are proved under some assumptions by means of a fixed point argument. Section 4 is devoted to the Eulerian current dynamics. In particular, we prove that the two formulations are equivalent. In particular, the well posedness of the Lagrangian formulation translates into the well posednes of the Eulerian formulation and viceversa.   In section 5, a result about continuous dependence on initial conditions is proved, that will be used later for proving a mean field result. A sequence of interacting curves (filaments) are defined  in section 6. These curves are  solutions of a system of differential equations (with a scaling $\alpha_{j}^N$),  that describe our flow of diffeomorphism. 
Here we are using a smooth kernel which  could be a mollified version of the Biot-Savart formula.
To this family of curves, we associate a  current defined in the vein of empirical measures.  We prove a mean field result when the number of filament $N\to\infty$. A similar result is also proved when the filaments are random in section 6.3. 

\section{Preliminaries on 1-currents\label{section preliminaries}}

Given $k,d,m\in\mathbb{N}$, we denote by $C_{b}^{k}(\mathbb{R}^{d}%
,\mathbb{R}^{m})$ the space of all functions $f:\mathbb{R}^{d}\rightarrow
\mathbb{R}^{m}$ that are of class $C^{k}$, bounded with all derivatives of
orders up to $k$. By $\mathcal{U}C_{b}^{3}(\mathbb{R}^{d},\mathbb{R}^{m})$ we
denote the subset of $C_{b}^{3}(\mathbb{R}^{d},\mathbb{R}^{m})$ of those
functions $f$ such that $f$, $Df$ and $D^{2}f$ are also uniformly continuous.

\subsection{Generalities\label{subsect generalities on currents}}

Currents of dimension 1 (called 1-currents here) are linear continuous
mappings on the space $C_{0}^{\infty}\left(  \mathbb{R}^{d},\mathbb{R}%
^{d}\right)  $ of smooth compact support vector fields of $\mathbb{R}^{d}$. In
the sequel we shall only consider 1-currents which are continuous in the
$C_{b}\left(  \mathbb{R}^{d},\mathbb{R}^{d}\right)  $ topology.

Moreover, consider the space $C_{b}\left(  \mathbb{R}^{d};\mathbb{R}%
^{d}\right)  $ of continuous and bounded vector fields on $\mathbb{R}^{d}$,
denote the uniform topology by $\left\Vert \cdot\right\Vert _{\infty}$ and
consider the following Banach space of 1-currents:%
\[
\mathcal{M}:\mathcal{=}C_{b}\left(  \mathbb{R}^{d};\mathbb{R}^{d}\right)
^{\prime}.
\]
The topology induced by the duality will be denoted by $\left\vert
\cdot\right\vert _{\mathcal{M}}$:%
\[
\left\vert \xi\right\vert _{\mathcal{M}}:=\sup_{\left\Vert \theta\right\Vert
_{\infty}\leq1}\left\vert \xi\left(  \theta\right)  \right\vert .
\]

We are interested in the weak topology too, essential to deal with
approximation by \textquotedblleft filaments\textquotedblright. We define%
\[
\left\Vert \xi\right\Vert =\sup\{\xi(\theta)\;|\;\Vert\theta\Vert_{\infty
}+\text{Lip}(\theta)\leq1\}
\]
where Lip$(\theta)$ is the Lipschitz constant of $\theta$. We set%

\[
d\left(  \xi,\xi^{\prime}\right)  =\left\Vert \xi-\xi^{\prime}\right\Vert
\]
for all $\xi,\xi^{\prime}\in\mathcal{M}$. The number $\left\Vert
\xi\right\Vert $ is well defined and%
\[
\left\Vert \xi\right\Vert \leq\left\vert \xi\right\vert _{\mathcal{M}}%
\]
and $d\left(  \xi,\xi^{\prime}\right)  $ satisfies the conditions of a
distance. Convergence in the metric space $\left(  \mathcal{M},d\right)  $
corresponds to weak convergence in $\mathcal{M}$ as dual to $C_{b}\left(
\mathbb{R}^{d} ;\mathbb{R}^{d}\right)  $. Recall the following fact:

\begin{lemma}
\label{lemma complete metric space}If $B$ is a closed ball in $\left(
\mathcal{M},\left\vert \cdot\right\vert _{\mathcal{M}}\right)  $, then
$\left(  B,d\right)  $ is a complete metric space.
\end{lemma}

\begin{proof}
Let $\{\xi_{n}\}_{n\geq0}$ be a Cauchy sequence in $(B,d)$. This is also a
Cauchy sequence in the dual space Lip$_{b}(\mathbb{R}^{d},\mathbb{R}%
^{d})^{\prime}$ with the dual operator norm. Hence it converges to some
$\xi\in$Lip$_{b}(\mathbb{R}^{d},\mathbb{R}^{d})^{\prime}$. Indeed
Lip$_{b}(\mathbb{R}^{d},\mathbb{R}^{d})$ is a Banach space and $\Vert
\cdot\Vert$ is the operator norm on his dual, which is complete.

Now we have an operator $\xi$ defined on Lip$_{b}(\mathbb{R}^{d}%
,\mathbb{R}^{d})$, we want to extend it to the bigger space $C_{b}\left(
\mathbb{R}^{d};\mathbb{R}^{d}\right)  $ and to show that this extension is a
limit to the sequence $\xi_{n}$ in the norm $\Vert\cdot\Vert$.

Given $\theta\in$Lip$_{b}(\mathbb{R}^{d},\mathbb{R}^{d})$, it holds, for every
$n\in\mathbb{N}$,
\[
|\xi(\theta)|\leq|(\xi-\xi_{n})(\theta)|+|\xi_{n}(\theta)|\leq\Vert\xi-\xi
_{n}\Vert(\Vert\theta\Vert_{\infty}+\text{Lip}(\theta))+R\Vert\theta
\Vert_{\infty}%
\]
where $R$ denotes the radius of $B$. Hence, as $n\rightarrow\infty$, it holds
$|\xi(\theta)|\leq R\Vert\theta\Vert_{\infty}$. We can thus apply Hahn-Banach
theorem to obtain a linear functional $\bar{\xi}$ defined on $C_{b}\left(
\mathbb{R}^{d};\mathbb{R}^{d}\right)  $ such that $\Vert\bar{\xi}\Vert\leq R$
and $\bar{\xi}\equiv\xi$ on Lip$_{b}(\mathbb{R}^{d},\mathbb{R}^{d})^{\prime}$.

It only remains to prove that $\xi_{n}$ converges to $\bar{\xi}$,
\[
\Vert\bar{\xi}-\xi_{n}\Vert=\sup\{\bar{\xi}(\theta) - \xi_n(\theta)\;|\;\Vert
\theta\Vert_{\infty}+\text{Lip}(\theta)\leq1\}
\]%
\[
=\sup\{\xi(\theta) - \xi_n(\theta)\;|\;\Vert\theta\Vert_{\infty}+\text{Lip}%
(\theta)\leq1\}=\Vert\xi-\xi_{n}\Vert\rightarrow0,\quad\text{as}%
\;n\rightarrow\infty.
\]

\end{proof}

We shall denote by $\mathcal{M}_{w}$ the space $\mathcal{M}$ endowed by the
metric $d$.

If $\xi\in\mathcal{M}$ and $K:\mathbb{R}^{d}\rightarrow\mathbb{R}^{d\times d}$
is a continuous bounded matrix-valued function, then $K\ast\xi$ is the vector
field in $\mathbb{R}^{d}$ with $i$-component given by%
\begin{equation}
\left(  K\ast\xi\right)  _{i}\left(  x\right)  =\left(  K_{i\cdot}\ast
\xi\right)  \left(  x\right)  :=\xi\left(  K_{i\cdot}\left(  x-\cdot\right)
\right)  \label{def convol}%
\end{equation}
where $K_{i\cdot}\left(  z\right)  $ is the vector $\left(  K_{ij}\left(
z\right)  \right)  _{j=1,...,d}$. We have
\[
\left\vert \left(  K\ast\xi\right)  \left(  x\right)  \right\vert
\leq\left\vert \xi\right\vert _{\mathcal{M}}\Vert K\Vert_{\infty}.
\]
If $K$, in addition, is also of class $C_{b}^{1}(\mathbb{R}^{d},\mathbb{R}%
^{m})$, then%
\begin{equation}
\left\vert \left(  K\ast\xi\right)  \left(  x\right)  \right\vert
\leq\left\Vert \xi\right\Vert \left(  \Vert K\Vert_{\infty}+\Vert
DK\Vert_{\infty}\right)  . \label{bound on convolution}%
\end{equation}

\subsection{Push-forward}

Let $\theta\in C_{b}\left(  \mathbb{R}^{d},\mathbb{R}^{d}\right)  $ be a
vector field (test function) and $\varphi:\mathbb{R}^{d}\rightarrow
\mathbb{R}^{d}$ be a map. When defined, the \textit{push-forward} of $\theta$
is%
\[
\left(  \varphi_{\sharp}\theta\right)  \left(  x\right)  =D\varphi\left(
x\right)  ^{T}\theta\left(  \varphi\left(  x\right)  \right)  .
\]
If $\varphi$ is of class $C^{1}\left(  \mathbb{R}^{d};\mathbb{R}^{d}\right)
$, then $\varphi_{\sharp}$ is a well defined bounded linear map from
$C_{b}\left(  \mathbb{R}^{d},\mathbb{R}^{d}\right)  $ to itself.

Given a curve $\gamma:\left[  0,1\right]  \rightarrow\mathbb{R}^{d}$ of class
$C^{1}$ ($W^{1,1}$ is sufficient), consider the current%
\[
\xi=\int_{0}^{1}\delta\left(  \cdot-\gamma\left(  \sigma\right)  \right)
\frac{d\gamma}{d\sigma}\left(  \sigma\right)  d\sigma
\]
namely the linear functional $\xi:C_{b}\left(  \mathbb{R}^{d},\mathbb{R}%
^{d}\right)  \rightarrow\mathbb{R}$ defined as
\[
\xi\left(  \theta\right)  =\int_{0}^{1}\left\langle \theta\left(
\gamma\left(  \sigma\right)  \right)  ,\frac{d\gamma}{d\sigma}\left(
\sigma\right)  \right\rangle _{\mathbb{R}^{d}}d\sigma.
\]

If $\varphi\in C^{1}\left(  \mathbb{R}^{d},\mathbb{R}^{d}\right)  ,$we define%
\[
\varphi_{\sharp}\xi:=\int_{0}^{1}\delta\left(  \cdot-\varphi\left(
\gamma\left(  \sigma\right)  \right)  \right)  D\varphi\left(  \gamma\left(
\sigma\right)  \right)  \frac{d\gamma}{d\sigma}\left(  \sigma\right)
d\sigma.
\]

\begin{remark}
Given $\gamma$ and $\varphi$, define the curve $\eta:\left[  0,1\right]
\rightarrow\mathbb{R}^{d}$ as $\eta\left(  \sigma\right)  =\varphi\left(
\gamma\left(  \sigma\right)  \right)  $. Notice that $\frac{d\eta}{d\sigma
}\left(  \sigma\right)  =D\varphi\left(  \gamma\left(  \sigma\right)  \right)
\frac{d\gamma}{d\sigma}\left(  \sigma\right)  $. Hence $\varphi_{\sharp}\xi$
is the current associated to the curve $\eta$.
\end{remark}

For this example of push-forward, we have the following relation:%
\[
\left(  \varphi_{\sharp}\xi\right)  \left(  \theta\right)  =\xi\left(
\varphi_{\sharp}\theta\right)  .
\]

Indeed,%
\begin{align*}
\left(  \varphi_{\sharp}\xi\right)  \left(  \theta\right)   &  =\int_{0}%
^{1}\left\langle \theta\left(  \varphi\left(  \gamma\left(  \sigma\right)
\right)  \right)  ,D\varphi\left(  \gamma\left(  \sigma\right)  \right)
\frac{d\gamma}{d\sigma}\left(  \sigma\right)  \right\rangle _{\mathbb{R}^{d}%
}d\sigma\\
&  =\int_{0}^{1}\left\langle D\varphi\left(  \gamma\left(  \sigma\right)
\right)  ^{T}\theta\left(  \varphi\left(  \gamma\left(  \sigma\right)
\right)  \right)  ,\frac{d\gamma}{d\sigma}\left(  \sigma\right)  \right\rangle
_{\mathbb{R}^{d}}d\sigma\\
&  =\int_{0}^{1}\left\langle \left(  \varphi_{\sharp}\theta\right)  \left(
\gamma\left(  \sigma\right)  \right)  ,\frac{d\gamma}{d\sigma}\left(
\sigma\right)  \right\rangle _{\mathbb{R}^{d}}d\sigma.
\end{align*}

Motivated by the previous computation (which is relevant by itself because
ultimately we want to deal with vortex filaments), given a 1-current $\xi
\in\mathcal{M}$ and a smooth map $\varphi:\mathbb{R}^{d}\rightarrow
\mathbb{R}^{d}$ we define the push-forward $\varphi_{\sharp}\xi$ as the
current%
\[
\left(  \varphi_{\sharp}\xi\right)  \left(  \theta\right)  :=\xi\left(
\varphi_{\sharp}\theta\right)  ,\qquad\theta\in C_{b}\left(  \mathbb{R}%
^{d},\mathbb{R}^{d}\right)  .
\]

We have seen above that $\varphi_{\sharp}\xi$ has a nice reformulation when
$\xi$ is associated to a smooth curve. Let us find a reformulation when $\xi$
is associated to a vector field. Thus, with little abuse of notations, let
$\xi: \mathbb{R}^{d}\rightarrow\mathbb{R}^{d}$ be an integrable vector field
and denote by $\xi$ the associated current defined as%
\[
\xi\left(  \theta\right)  =\int_{\mathbb{R}^{d}}\left\langle \theta\left(
x\right)  ,\xi\left(  x\right)  \right\rangle _{\mathbb{R}^{d}}dx.
\]

\begin{proposition}
\label{Prop push forward vector fields}Assume that $\varphi$ is a
diffeomorphism of $\mathbb{R}^{d}$ and $\xi$ is a vector field on
$\mathbb{R}^{d}$ in $\mathbb{R}^{d}$ of class $L^{1}$. Then $\varphi_{\sharp
}\xi$ is the following vector field in $\mathbb{R}^{d}$, of class $L^{1}$:
\[
\left(  \varphi_{\sharp}\xi\right)  \left(  x\right)  =D\varphi\left(
\varphi^{-1}\left(  x\right)  \right)  \xi\left(  \varphi^{-1}\left(
x\right)  \right)  \left\vert \det D\varphi^{-1}\left(  x\right)  \right\vert
.
\]

\end{proposition}

\begin{proof}
By definition we have%
\begin{align*}
\left(  \varphi_{\sharp}\xi\right)  \left(  \theta\right)   &  =\xi\left(
\varphi_{\sharp}\theta\right)  =\int_{\mathbb{R}^{d}}\left\langle
D\varphi\left(  x\right)  ^{T}\theta\left(  \varphi\left(  x\right)  \right)
,\xi\left(  x\right)  \right\rangle _{\mathbb{R}^{d}}dx\\
&  =\int_{\mathbb{R}^{d}}\left\langle \theta\left(  \varphi\left(  x\right)
\right)  ,D\varphi\left(  x\right)  \xi\left(  x\right)  \right\rangle
_{\mathbb{R}^{d}}dx\\
&  \overset{y=\varphi\left(  x\right)  }{=}\int_{\mathbb{R}^{d}}\left\langle
\theta\left(  y\right)  ,D\varphi\left(  \varphi^{-1}\left(  y\right)
\right)  \xi\left(  \varphi^{-1}\left(  y\right)  \right)  \right\rangle
_{\mathbb{R}^{d}}\left\vert \det D\varphi^{-1}\left(  y\right)  \right\vert
dy.
\end{align*}

\end{proof}

\section{Lagrangian current dynamics}

In order to prove that the nonlinear vector-valued PDE (\ref{PDE for currents}%
) with initial condition $\xi_{0}\in\mathcal{M}$, has unique local solutions
in the space of currents, we adopt a Lagrangian point of view:\ we examine the
ordinary differential equation%
\begin{equation}
\frac{dx_{t}}{dt}=\left(  K\ast\xi_{t}\right)  \left(  x_{t}\right)  ,
\label{system eq 1}%
\end{equation}
consider the flow of diffeomorphisms $\varphi^{t,K\ast\xi}$ generated by it
and take the push forward of $\xi_{0}$ under this flow:
\begin{equation}
\xi_{t}=\varphi_{\sharp}^{t,K\ast\xi}\xi_{0},\qquad t\in\left[  0,T\right]  .
\label{system eq 2}%
\end{equation}
The pair of equations (\ref{system eq 1})-(\ref{system eq 2}) defines a closed
system for $\left(  \xi_{t}\right)  _{t\in\left[  0,T\right]  }$ which, for
small $T$, has a unique solution. We shall prove then that current-valued
solutions of the PDE (\ref{PDE for currents}) are in one-to-one correspondence
with current-valued solutions of the flow system (\ref{system eq 1}%
)-(\ref{system eq 2}) and thus we get local existence and uniqueness for
(\ref{PDE for currents}).

Since the specific linear form $K\ast\xi_{t}$ for the drift of equation
(\ref{system eq 1}) is irrelevant, we replace it with a more general, possibly
non-linear, map. Thus we investigate a "flow equation" of the form%
\[
\xi_{t}=\varphi_{\sharp}^{t,B\left(  \xi\right)  }\xi_{0},\qquad t\in\left[
0,T\right]
\]
where $B\left(  \xi_{t}\right)  $ is a time-dependent vector field in
$\mathbb{R}^{d}$, associated to the time-dependent current $\xi_{t}$, and
$\varphi^{t,B\left(  \xi\right)  }$ is the flow associated to $B\left(
\xi\right)  $ by the equation%
\begin{equation}
\frac{dx_{t}}{dt}=B\left(  \xi_{t}\right)  \left(  x_{t}\right)  .
\label{ODE with B}%
\end{equation}

\subsection{Assumptions on the drift}

Let us discuss the general assumptions that we impose on the drift $B$ of
equation (\ref{ODE with B}). We assume
\begin{equation}
B:\mathcal{M}_{w}\rightarrow C_{b}^{2}\left(  \mathbb{R}^{d},\mathbb{R}%
^{d}\right)  \label{assumption on B}%
\end{equation}
to be a continuous map such that for every $\xi,\xi^{\prime}\in\mathcal{M}$%

\begin{equation}
\Vert B(\xi)\Vert_{C_{b}^{2}}\leq C_{B}\left(  \Vert\xi\Vert+1\right)
\label{B1}%
\end{equation}

\begin{equation}
\label{B2}\|B(\xi)- B(\xi^{\prime})\|_{\infty}\leq C_{B}\|\xi-\xi^{\prime}\|
\end{equation}

\begin{equation}
\Vert DB(\xi)-DB(\xi^{\prime})\Vert_{\infty}\leq C_{B}\Vert\xi-\xi^{\prime
}\Vert\label{B3}%
\end{equation}
We denote by $DB$ and $D^{2}B$ the derivatives of $B$ in the $x\in
\mathbb{R}^{3}$ variable.

Our main example of $B$ is the linear function
\[
B(\xi)=K\ast\xi
\]
where $K:\mathbb{R}^{d}\rightarrow\mathbb{R}^{d\times d}$ (see
(\ref{def convol})). The necessary regularity of $K$ is specified by next lemma.

\begin{lemma}
\label{lemma on K}Let $K\in\mathcal{U}C_{b}^{3}(\mathbb{R}^{d},\mathbb{R}%
^{d\times d})$. Then $B(\xi)=K\ast\xi$ maps continuously $\mathcal{M}$ in to
$C_{b}^{2}\left(  \mathbb{R}^{d},\mathbb{R}^{d}\right)  $ and satisfies
assumptions (\ref{B1})-(\ref{B3}).
\end{lemma}

\begin{proof}
Since $K\in C_{b}(\mathbb{R}^{d},\mathbb{R}^{d\times d})$, $K\ast
\xi:\mathbb{R}^{d}\rightarrow\mathbb{R}^{d}$ is a well defined function, for
every $\xi\in\mathcal{M}$. From (\ref{def convol}) and the uniform continuity
of $K$ it follows that $K\ast\xi$ is a continuous function: if $x_{n}%
\rightarrow x$, then $K_{i\cdot}\left(  x_{n}-\cdot\right)  \rightarrow
K_{i\cdot}\left(  x-\cdot\right)  $ uniformly, for every $i=1,...,d$. It is
bounded, since%
\begin{equation}
\left\vert \left(  K\ast\xi\right)  _{i}\left(  x\right)  \right\vert
\leq\left\Vert \xi\right\Vert \left(  \left\Vert K\right\Vert _{\infty
}+\left\Vert DK\right\Vert _{\infty}\right)  \label{uniform ineq in weak norm}%
\end{equation}
Moreover, the linear map $B:\mathcal{M}_{w}\rightarrow C_{b}\left(
\mathbb{R}^{d},\mathbb{R}^{d}\right)  $, just defined is continous in the weak
topology of $\mathcal{M}$: from the previous inequality it follows \
\[
\left\Vert \left(  K\ast\xi\right)  _{i}-\left(  K\ast\xi^{\prime}\right)
_{i}\right\Vert _{\infty}\leq\left\Vert \xi-\xi^{\prime}\right\Vert \left(
\left\Vert K\right\Vert _{\infty}+\left\Vert DK\right\Vert _{\infty}\right)
.
\]

Let us show that all the same facts extend to the first derivatives of
$K\ast\xi$. Since $DK$ is uniformly continuous and bounded, from%
\begin{align*}
&  \left\vert \frac{K_{ij}(x+\epsilon h)-K_{ij}(x)}{\epsilon}-DK_{ij}(x)\cdot
h\right\vert \\
&  =\left\vert \int_{0}^{1}\frac{1}{\epsilon}DK_{ij}((1-\alpha) x+\alpha
(x+\epsilon h))\cdot\epsilon h\;d\alpha-\int_{0}^{1}DK_{ij}(x)\cdot
h\;d\alpha\right\vert \\
&  \leq\int_{0}^{1}\left\vert DK_{ij}((1-\alpha) x+\alpha
(x+\epsilon h))
-DK_{ij}(x)\right\vert \;d\alpha
\end{align*}
it follows that the incremental ratio of $K_{ij}$ in a direction $h$ converges
uniformly to $DK_{ij}\cdot h$. From%
\[
\lim_{\epsilon\rightarrow0}\frac{1}{\epsilon}\xi(K_{i\cdot}(x+\epsilon
h-\cdot))-\xi(K_{i\cdot}(x-\cdot))=\lim_{\epsilon\rightarrow0}\xi\left(
\frac{K_{i\cdot}(x+\epsilon h-\cdot)-K_{i\cdot}(x-\cdot)}{\epsilon}\right)
\]
it follows that $K\ast\xi$ is differentiable at every point and%
\[
D\left(  K\ast\xi\right)  _{i}\left(  x\right)  \cdot h=\xi\left(  DK_{i\cdot
}\left(  x\right)  \cdot h\right)  .
\]
The arguments now are similar to those already exposed above and iterate: this
first derivatives are continuous bounded functions and $B:\mathcal{M}%
_{w}\rightarrow C_{b}^{1}\left(  \mathbb{R}^{d},\mathbb{R}^{d}\right)  $ is continuous.

Iterating again, based on the uniform continuity of $D^{2}K$ and the property
$K\in C_{b}^{3}$, we get that $B:\mathcal{M}_{w}\rightarrow C_{b}^{2}\left(
\mathbb{R}^{d},\mathbb{R}^{d}\right)  $ is well defined and continous.
Property (\ref{B1}) comes from (\ref{uniform ineq in weak norm}) and the
similar inequalities for first and second derivatives; the last one requires
$K$ of class $C_{b}^{3}$. Property (\ref{B2}) follows from
(\ref{uniform ineq in weak norm}). Finally, property (\ref{B3}) is proved
similarly, using the analogous bound on the second derivative.
\end{proof}

\subsection{Properties of the flow\label{section the flow}}

For any $b\in C\left(  \left[  0,T\right]  ;C_{b}^{2}\left(  \mathbb{R}%
^{d},\mathbb{R}^{d}\right)  \right)  $, consider the ODE in $\mathbb{R}^{d}$
\begin{equation}
X_{t}^{\prime}=b\left(  t,X_{t}\right)  \label{ode}%
\end{equation}
and denote by $\varphi^{t,b}\left(  x\right)  $ the associated flow. It is
differentiable and%
\begin{equation}
\frac{d}{dt}D\varphi^{t,b}\left(  x\right)  =Db\left(  t,\varphi^{t,b}\left(
x\right)  \right)  D\varphi^{t,b}\left(  x\right)  ,\qquad D\varphi
^{0,b}\left(  x\right)  =Id. \label{diff}%
\end{equation}
In the sequel we shall denote by $B\left(  \xi\right)  $ also the function
$t\mapsto B\left(  \xi_{t}\right)  $. Moreover, we write%
\[
\left\Vert \xi\right\Vert _{T}:=\sup_{0\leq t\leq T}\left\Vert \xi
_{T}\right\Vert .
\]
The computations in the proof of the following lemma are classical; however,
it is important for Theorem \ref{thm maximal} below that we carefully
make the estimates (\ref{Lip-phi}) and (\ref{Lip-Dphi}) depend only on one of
the two current-valued processes, say $\xi$;\ this asymmetric dependence is
less obvious, although common to other problems like the theorems of
weak-strong uniqueness.

\begin{lemma}
\label{flow} If $\xi\in C\left(  \left[  0,T\right]  ;\mathcal{M}_{w}\right)
$, then the flow $\varphi^{t,B(\xi)}:\mathbb{R}^{d}\rightarrow\mathbb{R}^{d}$
is twice differentiable and satisfies, for all $t\in\left[  0,T\right]  $,
\begin{equation}
\left\Vert D\varphi^{t,B\left(  \xi\right)  }\right\Vert _{\infty}\leq
e^{C_{B}T(\left\Vert \xi\right\Vert _{T}+1)} \label{Dphi}%
\end{equation}

\begin{equation}
\left\Vert \varphi^{t,B\left(  \xi\right)  }-\varphi^{t,B\left(  \xi^{\prime
}\right)  }\right\Vert _{\infty}\leq C_{B}Te^{C_{B}T(\left\Vert \xi\right\Vert
_{T}+1)}\left\Vert \xi-\xi^{\prime}\right\Vert _{T} \label{Lip-phi}%
\end{equation}
and for every $x\in\mathbb{R}^{d}$%
\[
\left\Vert D\varphi^{t,B\left(  \xi\right)  }-D\varphi^{t,B\left(  \xi
^{\prime}\right)  }\right\Vert _{\infty}%
\]%
\begin{equation}
\leq C_{B}Te^{C_{B}\left(  T+1\right)  (\left\Vert \xi\right\Vert
_{T}+\left\Vert \xi^{\prime}\right\Vert _{T}+2)}\left(  1+C_{B}T(\left\Vert
\xi\right\Vert _{T}+1)e^{C_{B}T(\left\Vert \xi\right\Vert _{T}+1)}\right)
\Vert\xi-\xi^{\prime}\Vert_{T}. \label{Lip-Dphi}%
\end{equation}

Moreover, for every $x,y\in\mathbb{R}^{d}$,
\begin{equation}
|D\varphi^{t,B\left(  \xi\right)  }\left(  x\right)  -D\varphi^{t,B\left(
\xi\right)  }\left(  y\right)  |\leq C_{B}T\left(  \left\Vert \xi\right\Vert
_{T}+1\right)  e^{C_{B}\left(  2T+1\right)  \left(  \left\Vert \xi\right\Vert
_{T}+1\right)  }\left\vert x-y\right\vert . \label{Dphi_Lipschitz}%
\end{equation}

\end{lemma}

\begin{proof}
We have, from $\frac{d}{dt}D\varphi^{t,B\left(  \xi\right)  }\left(  x\right)
=DB\left(  \xi_{t}\right)  \left(  \varphi^{t,B\left(  \xi\right)  }\left(
x\right)  \right)  D\varphi^{t,B\left(  \xi\right)  }\left(  x\right)  $,
\[
\left\vert D\varphi^{t,B\left(  \xi\right)  }\left(  x\right)  \right\vert
\leq e^{\int_{0}^{t}\left\vert DB\left(  \xi_{s}\right)  \left(
\varphi^{s,B\left(  \xi\right)  }\left(  x\right)  \right)  \right\vert ds}%
\]%
\[
\left\Vert D\varphi^{t,B\left(  \xi\right)  }\right\Vert _{\infty}\leq
e^{\int_{0}^{t}\left\Vert DB\left(  \xi_{s}\right)  \right\Vert _{\infty}ds}.
\]
Now, using the assumption (\ref{B1}) on $B$ we get (\ref{Dphi}).

For the estimate (\ref{Lip-phi}), notice that
\begin{align*}
\frac{d}{dt}\left(  \varphi^{t,B\left(  \xi\right)  }\left(  x\right)
-\varphi^{t,B\left(  \xi^{\prime}\right)  }\left(  x\right)  \right)   &
=B\left(  \xi_{t}\right)  \left(  \varphi^{t,B\left(  \xi\right)  }\left(
x\right)  \right)  -B\left(  \xi_{t}^{\prime}\right)  \left(  \varphi
^{t,B\left(  \xi^{\prime}\right)  }\left(  x\right)  \right) \\
&  =B\left(  \xi_{t}\right)  \left(  \varphi^{t,B\left(  \xi\right)  }\left(
x\right)  \right)  -B\left(  \xi_{t}\right)  \left(  \varphi^{t,B\left(
\xi^{\prime}\right)  }\left(  x\right)  \right) \\
&  +B\left(  \xi_{t}\right)  \left(  \varphi^{t,B\left(  \xi^{\prime}\right)
}\left(  x\right)  \right)  -B\left(  \xi_{t}^{\prime}\right)  \left(
\varphi^{t,B\left(  \xi^{\prime}\right)  }\left(  x\right)  \right)
\end{align*}
hence%
\begin{align*}
\left\vert \varphi^{t,B\left(  \xi\right)  }\left(  x\right)  -\varphi
^{t,B\left(  \xi^{\prime}\right)  }\left(  x\right)  \right\vert  &  \leq
\int_{0}^{t}\left\Vert DB\left(  \xi_{s}\right)  \right\Vert _{\infty
}\left\vert \varphi^{s,B\left(  \xi\right)  }\left(  x\right)  -\varphi
^{s,B\left(  \xi^{\prime}\right)  }\left(  x\right)  \right\vert ds\\
&  +\int_{0}^{t}\left\Vert B\left(  \xi_{s}\right)  -B\left(  \xi_{s}^{\prime
}\right)  \right\Vert _{\infty}ds.
\end{align*}
Hence, using Gronwall's lemma we get that%
\[
\left\vert \varphi^{t,B\left(  \xi\right)  }\left(  x\right)  -\varphi
^{t,B\left(  \xi^{\prime}\right)  }\left(  x\right)  \right\vert \leq\int%
_{0}^{t}\left\Vert B\left(  \xi_{s}\right)  -B\left(  \xi_{s}^{\prime}\right)
\right\Vert _{\infty}e^{\int_{s}^{t}\left\Vert DB\left(  \xi_{r}\right)
\right\Vert _{\infty}dr}ds.
\]
Now, using again assumptions (\ref{B1}) and (\ref{B2}), we deduce
(\ref{Lip-phi}).

Now, let us prove (\ref{Lip-Dphi}). Let us notice that
\begin{align*}
&  \frac{d}{dt}\left(  D\varphi^{t,B\left(  \xi\right)  }\left(  x\right)
-D\varphi^{t,B\left(  \xi^{\prime}\right)  }\left(  x\right)  \right) \\
&  =DB\left(  \xi_{t}\right)  \left(  \varphi^{t,B\left(  \xi\right)  }\left(
x\right)  \right)  D\varphi^{t,B\left(  \xi\right)  }\left(  x\right)
-DB\left(  \xi_{t}^{\prime}\right)  \left(  \varphi^{t,B\left(  \xi^{\prime
}\right)  }\left(  x\right)  \right)  D\varphi^{t,B\left(  \xi^{\prime
}\right)  }\left(  x\right) \\
&  =DB\left(  \xi_{t}\right)  \left(  \varphi^{t,B\left(  \xi\right)  }\left(
x\right)  \right)  D\varphi^{t,B\left(  \xi\right)  }\left(  x\right)
-DB\left(  \xi_{t}\right)  \left(  \varphi^{t,B\left(  \xi\right)  }\left(
x\right)  \right)  D\varphi^{t,B\left(  \xi^{\prime}\right)  }\left(  x\right)
\\
&  +DB\left(  \xi_{t}\right)  \left(  \varphi^{t,B\left(  \xi\right)  }\left(
x\right)  \right)  D\varphi^{t,B\left(  \xi^{\prime}\right)  }\left(
x\right)  -DB\left(  \xi_{t}^{\prime}\right)  \left(  \varphi^{t,B\left(
\xi^{\prime}\right)  }\left(  x\right)  \right)  D\varphi^{t,B\left(
\xi^{\prime}\right)  }\left(  x\right)  .
\end{align*}

For the first term,%
\begin{align*}
&  \left\vert DB\left(  \xi_{t}\right)  \left(  \varphi^{t,B\left(
\xi\right)  }\left(  x\right)  \right)  D\varphi^{t,B\left(  \xi\right)
}\left(  x\right)  -DB\left(  \xi_{t}\right)  \left(  \varphi^{t,B\left(
\xi\right)  }\left(  x\right)  \right)  D\varphi^{t,B\left(  \xi^{\prime
}\right)  }\left(  x\right)  \right\vert \\
&  \leq\left\Vert DB\left(  \xi_{t}\right)  \right\Vert _{\infty}\left\vert
D\varphi^{t,B\left(  \xi\right)  }\left(  x\right)  -D\varphi^{t,B\left(
\xi^{\prime}\right)  }\left(  x\right)  \right\vert .
\end{align*}

For the second term,%
\begin{align*}
&  \left\vert DB(\xi_{t})\left(  \varphi^{t,B(\xi)}(x)\right)  D\varphi
^{t,B(\xi^{\prime})}\left(  x\right)  -DB(\xi_{t}^{\prime})\left(
\varphi^{t,B(\xi^{\prime})}(x)\right)  D\varphi^{t,B(\xi^{\prime})}\left(
x\right)  \right\vert \\
&  \leq\left\vert DB(\xi_{t})\left(  \varphi^{t,B(\xi)}(x)\right)  -DB(\xi
_{t}^{\prime})\left(  \varphi^{t,B(\xi^{\prime})}(x)\right)  \right\vert
\left\vert D\varphi^{t,B(\xi^{\prime})}\left(  x\right)  \right\vert \\
&  \leq\left\vert DB(\xi_{t})\left(  \varphi^{t,B(\xi)}(x)\right)  -DB(\xi
_{t})\left(  \varphi^{t,B(\xi^{\prime})}(x)\right)  \right\vert \left\vert
D\varphi^{t,B(\xi^{\prime})}\left(  x\right)  \right\vert \\
&  +\left\vert DB(\xi_{t})\left(  \varphi^{t,B(\xi^{\prime})}(x)\right)
-DB(\xi_{t}^{\prime})\left(  \varphi^{t,B(\xi^{\prime})}(x)\right)
\right\vert \left\vert D\varphi^{t,B(\xi^{\prime})}\left(  x\right)
\right\vert \\
&  \leq\left\Vert D^{2}B(\xi_{t})\right\Vert _{\infty}\left\vert
\varphi^{t,B(\xi)}(x)-\varphi^{t,B(\xi^{\prime})}(x)\right\vert \left\vert
D\varphi^{t,B(\xi^{\prime})}\left(  x\right)  \right\vert \\
&  +\left\Vert DB(\xi_{t})-DB(\xi_{t}^{\prime})\right\Vert _{\infty}\left\vert
D\varphi^{t,B(\xi^{\prime})}\left(  x\right)  \right\vert .
\end{align*}
Hence, using assumption (\ref{B1}), (\ref{B3}) and the estimates (\ref{Dphi})
and (\ref{Lip-phi}), we get that%
\begin{align*}
\left\vert D\varphi^{t,B\left(  \xi\right)  }\left(  x\right)  -D\varphi
^{t,B\left(  \xi^{\prime}\right)  }\left(  x\right)  \right\vert  &  \leq
C_{B}(\left\Vert \xi\right\Vert _{T}+1)\int_{0}^{t}\left\vert D\varphi
^{s,B\left(  \xi\right)  }\left(  x\right)  -D\varphi^{s,B\left(  \xi^{\prime
}\right)  }\left(  x\right)  \right\vert ds\\
&  +C_{B}Te^{C_{B}T(\left\Vert \xi^{\prime}\right\Vert _{T}+1)}\left(
1+C_{B}T(\left\Vert \xi\right\Vert _{T}+1)e^{C_{B}T(\left\Vert \xi\right\Vert
_{T}+1)}\right)  \Vert\xi-\xi^{\prime}\Vert_{T}%
\end{align*}
which implies, by Gronwall's lemma,
\[
\left\vert D\varphi^{t,B\left(  \xi\right)  }\left(  x\right)  -D\varphi
^{t,B\left(  \xi^{\prime}\right)  }\left(  x\right)  \right\vert \leq
e^{T C_{B}(\left\Vert \xi\right\Vert _{T}+1)}C_{B}Te^{C_{B}T(\left\Vert
\xi^{\prime}\right\Vert _{T}+1)}\left(  1+C_{B}T(\left\Vert \xi\right\Vert
_{T}+1)e^{C_{B}T(\left\Vert \xi\right\Vert _{T}+1)}\right)  \Vert\xi
-\xi^{\prime}\Vert_{T}.
\]
It is left to prove (\ref{Dphi_Lipschitz}).
\begin{align*}
|D\varphi^{t,B\left(  \xi\right)  }\left(  x\right)   &  -D\varphi^{t,B\left(
\xi\right)  }\left(  y\right)  |\\
\leq &  \int_{0}^{t}\left\vert DB(\xi_{s})(\varphi^{s,B(\xi)}(x))D\varphi
^{t,B(\xi)}(x)-DB(\xi_{s})(\varphi^{s,B(\xi)}(y))D\varphi^{t,B(\xi)}\left(
y\right)  \right\vert ds\\
\leq &  \int_{0}^{t}\left\vert DB(\xi_{s})(\varphi^{s,B(\xi)}(x))D\varphi
^{s,B(\xi)}(x)-DB(\xi_{s})(\varphi^{s,B(\xi)}(x))D\varphi^{s,B(\xi)}\left(
y\right)  \right\vert \\
&  +\left\vert DB(\xi_{s})(\varphi^{s,B(\xi)}(x))D\varphi^{s,B(\xi)}%
(y)-DB(\xi_{s})(\varphi^{s,B(\xi)}(y))D\varphi^{s,B(\xi)}\left(  y\right)
\right\vert ds\\
\leq &  \sup_{s\in\lbrack0,t]}\Vert DB(\xi_{s})\Vert_{\infty}\int_{0}%
^{t}\left\vert D\varphi^{s,B\left(  \xi\right)  }\left(  x\right)
-D\varphi^{s,B\left(  \xi\right)  }\left(  y\right)  \right\vert ds\\
&  +t\sup_{s\in\lbrack0,t]}\left(  \left\Vert D\varphi^{s,B(\xi)}\right\Vert
_{\infty}^{2}\left\Vert D^{2}B(\xi_{s})\right\Vert _{\infty}\right)
\left\vert x-y\right\vert .
\end{align*}
We now apply Gronwall's Lemma and we get
\[
|D\varphi^{t,B\left(  \xi\right)  }\left(  x\right)  -D\varphi^{t,B\left(
\xi\right)  }\left(  y\right)  |\leq T\sup_{s\in\lbrack0,T]}\left(  \left\Vert
D\varphi^{s,B(\xi)}\right\Vert _{\infty}^{2}\left\Vert D^{2}B(\xi
_{s})\right\Vert _{\infty}\right)  e^{T\sup_{s\in\lbrack0,T]}\Vert DB(\xi
_{s})\Vert_{\infty}}\left\vert x-y\right\vert .
\]
Now, using (\ref{flow}) and (\ref{B1}) we get (\ref{Dphi_Lipschitz}).
\end{proof}

\subsection{Well posedness of the flow equation}

We are now ready to consider the closed loop $\xi\mapsto\varphi^{\cdot
,B\left(  \xi\right)  }\mapsto\xi_{t}:=\varphi_{\sharp}^{t,B\left(
\xi\right)  }\xi_{0}$, namely the equation:
\begin{equation}
\xi_{t}=\varphi_{\sharp}^{t,B\left(  \xi\right)  }\xi_{0},\qquad t\in\left[
0,T\right]  . \label{xi}%
\end{equation}
Let us prove it has a unique solution in the space $C([0,T];\mathcal{M})$ by
using a fixed point argument. Indeed, let $\xi_{0}\in\mathcal{M}$ be the
initial current, at time $t=0$. Given $\xi=\left(  \xi_{t}\right)
_{t\in\left[  0,T\right]  }\in C\left(  [0,T];\mathcal{M}_{w}\right)  $, let
$\Gamma\left(  \xi\right)  =\eta=\left(  \eta_{t}\right)  _{t\in\left[
0,T\right]  }$ be the time-dependent current defined as%
\begin{equation}
\eta_{t}=\varphi_{\sharp}^{t,B\left(  \xi\right)  }\xi_{0},\qquad t\in\left[
0,T\right]  . \label{eta}%
\end{equation}

\begin{lemma}
\label{lemma bounded in bounded}Given $\xi_{0}\in\mathcal{M}$, set
$R=2\left\vert \xi_{0}\right\vert _{\mathcal{M}}$. Then there exists
$T_{R}^{0}>0$, depending only on $R$, such that $\Gamma\left(  B_{R}\right)
\subset B_{R}$, where $B_{R}$ is the set of all $\xi=\left(  \xi_{t}\right)
_{t\in\left[  0,T\right]  }\in C\left(  0,T;\mathcal{M}\right)  $ such that
$\sup_{t\in\lbrack0,T]}\left\vert \xi_{t}\right\vert _{\mathcal{M}}\leq R$.
Similarly if $B_{R}$ is defined by the norm $C\left(  [0,T];\mathcal{M}%
_{w}\right)  $, where $\mathcal{M}_{w}$ is endowed with the norm $\left\Vert
\cdot\right\Vert $.
\end{lemma}

\begin{proof}
First we prove the first statement. To do it, we must estimate the strong norm
of $\eta=\Gamma\left(  \xi\right)  $:%
\begin{align*}
\left\vert \varphi_{\sharp}^{t,B\left(  \xi\right)  }\xi_{0}\left(
\theta\right)  \right\vert  &  =\left\vert \xi_{0}\left(  \varphi_{\sharp
}^{t,B\left(  \xi\right)  }\theta\right)  \right\vert \leq\left\vert \xi
_{0}\right\vert _{\mathcal{M}}\left\Vert \varphi_{\sharp}^{t,B\left(
\xi\right)  }\theta\right\Vert _{\infty}=\left\vert \xi_{0}\right\vert
_{\mathcal{M}}\left\Vert D\varphi^{t,B\left(  \xi\right)  }\left(
\cdot\right)  ^{T}\theta\left(  \varphi^{t,B\left(  \xi\right)  }\left(
\cdot\right)  \right)  \right\Vert _{\infty}\\
&  \leq \left\vert \xi_{0}\right\vert _{\mathcal{M}} \left\Vert D\varphi^{t,B\left(  \xi\right)  }\right\Vert _{\infty
}\left\Vert \theta\left(  \varphi^{t,B\left(  \xi\right)  }\left(
\cdot\right)  \right)  \right\Vert _{\infty}\leq 
\left\vert \xi_{0}\right\vert _{\mathcal{M}} \left\Vert D\varphi
^{t,B\left(  \xi\right)  }\right\Vert _{\infty}\left\Vert \theta\right\Vert
_{\infty}%
\end{align*}
which implies that%
\begin{equation}
\sup_{t\in\left[  0,T\right]  }\left\vert \eta_{t}\right\vert _{\mathcal{M}%
}\leq\left\vert \xi_{0}\right\vert _{\mathcal{M}}\sup_{t\in\left[  0,T\right]
}\left\Vert D\varphi^{t,B\left(  \xi\right)  }\right\Vert _{\infty}.
\label{estimate strong topology}%
\end{equation}
Using (\ref{Dphi}) and $\Vert\xi\Vert\leq|\xi|_{\mathcal{M}}\leq R$, we get
that
\[
\sup_{t\in\left[  0,T\right]  }\left\vert \eta_{t}\right\vert _{\mathcal{M}%
}\leq\frac{R}{2}e^{C_{B}(R+1)T}.
\]
If $T$ satisfies $e^{C_{B}(R+1)T}\leq2$, we get $\Gamma\left(  B_{R}\right)
\subset B_{R}$ and the proof is complete.

To prove the second statement we first see, from the definition of the norm
$\Vert\cdot\Vert$, that it holds
\[
\sup_{t\in\left[  0,T\right]  }\left\Vert \eta_{t}\right\Vert \leq\sup
_{t\in\left[  0,T\right]  }\left\Vert \varphi_{\sharp}^{t,B\left(  \xi\right)
}\xi_{0}\right\Vert =\sup_{t\in\lbrack0,T]}\sup\{\xi_{0}(\varphi_{\sharp
}^{t,B\left(  \xi\right)  }\theta)\;|\;\Vert\theta\Vert_{\infty}%
+Lip(\theta)\leq1\}
\]
Now, proceeding as in the previous part, we estimate $|\xi_{0}(\varphi
_{\#}^{t,B(\xi)}\theta)|$ and the prove follows in the same way.
\end{proof}

\begin{theorem}
\label{Thm existence and uniqueness}For every $\xi_{0}\in\mathcal{M}$, there
is $T_{R}>0$, depending only on $\left\vert \xi_{0}\right\vert _{\mathcal{M}}%
$, such that there exists a unique solution $\xi$ of the flow equation
(\ref{xi}) in $C\left(  \left[  0,T_{R}\right]  ;\mathcal{M}\right)  $.
\end{theorem}

\begin{proof}
\textbf{Step 1}. Let $R=2\left\vert \xi_{0}\right\vert _{\mathcal{M}}$ and
$T_{R}^{0}$ be given by the previous lemma;\ let $T\in\left[  0,T_{R}%
^{0}\right]  $ to be chosen below. Let $\xi,\xi^{\prime}\in C\left(
[0,T];\mathcal{M}_{w}\right)  $, $\eta=\Gamma\left(  \xi\right)  ,\eta
^{\prime}=\Gamma\left(  \xi^{\prime}\right)  $. We have, for every Lipschitz
function $\theta$,
\begin{equation}
|(\eta_{t}-\eta_{t}^{\prime})(\theta)|\leq|\xi_{0}|_{\mathcal{M}}\Vert
\varphi_{\sharp}^{t,B\left(  \xi\right)  }\theta-\varphi_{\sharp}^{t,B\left(
\xi^{\prime}\right)  }\theta\Vert_{\infty} \label{new 1}%
\end{equation}
Now%
\begin{align*}
\left\Vert \varphi_{\sharp}^{t,B\left(  \xi\right)  }\theta-\varphi_{\sharp
}^{t,B\left(  \xi^{\prime}\right)  }\theta\right\Vert _{\infty}  &
=\left\Vert D\varphi^{t,B\left(  \xi\right)  }\left(  \cdot\right)  ^{T}%
\theta\left(  \varphi^{t,B\left(  \xi\right)  }\left(  \cdot\right)  \right)
-D\varphi^{t,B\left(  \xi^{\prime}\right)  }\left(  \cdot\right)  ^{T}%
\theta\left(  \varphi^{t,B\left(  \xi^{\prime}\right)  }\left(  \cdot\right)
\right)  \right\Vert _{\infty}\\
&  \leq\left\Vert D\varphi^{t,B\left(  \xi\right)  }\left(  \cdot\right)
^{T}\theta\left(  \varphi^{t,B\left(  \xi\right)  }\left(  \cdot\right)
\right)  -D\varphi^{t,B\left(  \xi\right)  }\left(  \cdot\right)  ^{T}%
\theta\left(  \varphi^{t,B\left(  \xi^{\prime}\right)  }\left(  \cdot\right)
\right)  \right\Vert _{\infty}\\
&  +\left\Vert D\varphi^{t,B\left(  \xi\right)  }\left(  \cdot\right)
^{T}\theta\left(  \varphi^{t,B\left(  \xi^{\prime}\right)  }\left(
\cdot\right)  \right)  -D\varphi^{t,B\left(  \xi^{\prime}\right)  }\left(
\cdot\right)  ^{T}\theta\left(  \varphi^{t,B\left(  \xi^{\prime}\right)
}\left(  \cdot\right)  \right)  \right\Vert _{\infty}\\
&  \leq\left\Vert D\varphi^{t,B\left(  \xi\right)  }\right\Vert _{\infty
}\left\Vert D\theta\right\Vert _{\infty}\left\Vert \varphi^{t,B\left(
\xi\right)  }-\varphi^{t,B\left(  \xi^{\prime}\right)  }\right\Vert _{\infty
}\\
&  +\left\Vert \theta\right\Vert _{\infty}\left\Vert D\varphi^{t,B\left(
\xi\right)  }-D\varphi^{t,B\left(  \xi^{\prime}\right)  }\right\Vert _{\infty
}.
\end{align*}
By definition, $\Vert\eta_{t}-\eta_{t}^{\prime}\Vert$ is less than or equal to
the supremum of (\ref{new 1}), taken over the Lipschitz functions $\theta$
such that $\Vert\theta\Vert_{\infty}+$Lip$(\theta)\leq1$. Hence,%
\begin{align*}
&  \sup_{t\in\left[  0,T\right]  }\left\Vert \eta_{t}-\eta_{t}^{\prime
}\right\Vert \\
&  \leq\left\vert \xi_{0}\right\vert _{\mathcal{M}}\sup_{t\in\left[
0,T\right]  }\left\Vert D\varphi^{t,B\left(  \xi\right)  }\right\Vert
_{\infty}\left\Vert \varphi^{t,B\left(  \xi\right)  }-\varphi^{t,B\left(
\xi^{\prime}\right)  }\right\Vert _{\infty}\\
&  +\left\vert \xi_{0}\right\vert _{\mathcal{M}}\sup_{t\in\left[  0,T\right]
}\left\Vert D\varphi^{t,B\left(  \xi\right)  }-D\varphi^{t,B\left(
\xi^{\prime}\right)  }\right\Vert _{\infty}%
\end{align*}
Hence, using (\ref{Dphi})
\begin{align*}
&  \sup_{t\in\left[  0,T\right]  }\left\Vert \eta_{t}-\eta_{t}^{\prime
}\right\Vert \\
&  \leq e^{C_{B}T(R+1)}\left\vert \xi_{0}\right\vert _{\mathcal{M}}\sup
_{t\in\left[  0,T\right]  }\left\Vert \varphi^{t,B\left(  \xi\right)
}-\varphi^{t,B\left(  \xi^{\prime}\right)  }\right\Vert _{\infty}+\left\vert
\xi_{0}\right\vert _{\mathcal{M}}\sup_{t\in\left[  0,T\right]  }\left\Vert
D\varphi^{t,B\left(  \xi\right)  }-D\varphi^{t,B\left(  \xi^{\prime}\right)
}\right\Vert _{\infty}.
\end{align*}

\textbf{Step 2}. Using the estimates given in lemma (\ref{flow}) and
summarizing
\begin{align*}
&  \sup_{t\in\left[  0,T\right]  }\left\Vert \eta_{t}-\eta_{t}^{\prime
}\right\Vert \\
&  \leq C_{B}Te^{2C_{B}T(R+1)}\left\vert \xi_{0}\right\vert _{\mathcal{M}}%
\sup_{t\in\left[  0,T\right]  }\left\Vert \xi_{t}-\xi_{t}^{\prime}\right\Vert
\\
&  +\left\vert \xi_{0}\right\vert _{\mathcal{M}}C_{B}Te^{C_{B}\left(
3T+2\right)  (R+1)}\left(  1+C_{B}T(R+1)\right)  \sup_{t\in\left[  0,T\right]
}\left\Vert \xi_{t}-\xi_{t}^{\prime}\right\Vert \\
&  \leq C_{B}T\left\vert \xi_{0}\right\vert _{\mathcal{M}}\left(
e^{C_{B}\left(  3T+2\right)  (R+1)}\left(  2+C_{B}T(R+1)\right)  \right)
\sup_{t\in\left[  0,T\right]  }\left\Vert \xi_{t}-\xi_{t}^{\prime}\right\Vert
.
\end{align*}
Therefore, for $T$ small enough, $\Gamma$ is a contraction in $C\left(
\left[  0,T\right]  ;\mathcal{M}_{w}\right)  $.

Recall now lemma \ref{lemma complete metric space}. The space of currents of
class $C\left(  [0,T];\mathcal{M}_{w}\right)  $ with $\sup_{t\in\lbrack
0,T]}\left\vert \xi_{t}\right\vert _{\mathcal{M}}\leq R$ is a complete metric
space, and $\Gamma$ is a contraction in this space, for $T$ small enough. It
follows that there exists a unique fixed point of $\Gamma$ in this space.
Finally, the fixed point is also in $C\left(  [0,T];\mathcal{M}\right)  $
since the output of the push forward is in this space.
\end{proof}

\subsection{Maximal solutions}

Given $\xi_{0}\in\mathcal{M}$, let $\Upsilon_{\xi_{0}}$ be the set of all
$T>0$ such that there exists a unique current-valued solution on $\left[
0,T\right]  $ for the flow equation (\ref{xi}) with initial condition $\xi
_{0}$. Let $T_{\xi_{0}}=\sup\Upsilon_{\xi_{0}}$; on $[0,T_{\xi_{0}})$ a unique
solution $\xi$ is well defined;\ it is called the maximal solution. We have
$\xi\in C([0,T_{\xi_{0}});\mathcal{M})$. An easy fact is:

\begin{lemma}
If $T_{\xi_{0}}<\infty$, then
\[
\lim_{t\rightarrow T_{\xi_{0}}^{-}}\left\vert \xi_{t}\right\vert
_{\mathcal{M}}=+\infty.
\]

\end{lemma}

\begin{proof}
We prove the claim by contradiction. Assume there is a sequence $t_{n}%
\rightarrow T_{\xi_{0}}^{-}$ and a constant $C>0$ such that $\left\vert
\xi_{t_{n}}\right\vert _{\mathcal{M}}\leq C$ for every $n\in\mathbb{N}$. We
may apply the existence and uniqueness theorem on the time interval $\left[
t_{n},t_{n}+T\right]  $ where $T$ depends only on $C$. Hence a unique solution
exists up to time $t_{n}+T$. For large $n$ this contradicts the definition of
$T_{\xi_{0}}$ , if it is finite.
\end{proof}

 Taking into account that we
only have $\left\Vert \xi\right\Vert \leq\left\vert \xi\right\vert
_{\mathcal{M}}$, we have the following interesting criterion.
\begin{theorem}
\label{thm maximal}If $T_{\xi_{0}}<\infty$, then
\[
\lim_{t\rightarrow T_{\xi_{0}}^{-}}\left\Vert \xi_{t}\right\Vert =+\infty.
\]

\end{theorem}

\begin{proof}
For $t\in\lbrack0,T_{\xi_{0}})$ we have (the proof is the same as estimate
(\ref{estimate strong topology}) in Lemma \ref{lemma bounded in bounded})%
\begin{equation}
\left\vert \xi_{t}\right\vert _{\mathcal{M}}\leq\left\vert \xi_{0}\right\vert
_{\mathcal{M}}\left\Vert D\varphi^{t,B\left(  \xi\right)  }\right\Vert
_{\infty}\leq\left\vert \xi_{0}\right\vert _{\mathcal{M}}e^{C_{B}t(\left\Vert
\xi\right\Vert _{t}+1)} \label{estimate between norms}%
\end{equation}
having used (\ref{Dphi}) in the last step. Hence ($\xi_{0}\neq0$, otherwise
$T_{\xi_{0}}=+\infty$), for $t\in(0,T_{\xi_{0}})$
\[
\left\Vert \xi\right\Vert _{t}\geq\frac{1}{C_{B}t}\log\frac{\left\vert \xi
_{t}\right\vert _{\mathcal{M}}}{\left\vert \xi_{0}\right\vert _{\mathcal{M}}%
}-1.
\]
From $\lim_{t\rightarrow T_{\xi_{0}}^{-}}\left\vert \xi_{t}\right\vert
_{\mathcal{M}}=+\infty$ it follows $\lim_{t\rightarrow T_{\xi_{0}}^{-}%
}\left\Vert \xi_{t}\right\Vert =+\infty$.
\end{proof}

\section{Eulerian current dynamics}

Given an operator $B$ with the assumptions exposed at the beginning of Section
\ref{section the flow}, and taking values in the set of divergence free vector
fields, consider the non-linear PDE%
\begin{equation}
\left\{
\begin{array}
[c]{l}%
\frac{\partial\xi}{\partial t}+\operatorname{div}\left(  B\left(  \xi\right)
\xi\right)  =(\xi\cdot\nabla)B\left(  \xi\right) \\
\xi(0)=\xi_{0}.
\end{array}
\right.  \label{PDE nonlinear with B}%
\end{equation}

\begin{definition}
\label{solution} We say that $\xi\in C([0,T];\mathcal{M})$ is a current-valued
solution for the PDE (\ref{PDE nonlinear with B}) if for every $\theta\in
C_{b}^{1}(\mathbb{R}^{d};\mathbb{R}^{d})$ and every $t\in\lbrack0,T]$, it
satisfies%
\begin{equation}
\xi_{t}\left(  \theta\right)  -\int_{0}^{t}\xi_{s}\left(  D\theta\cdot
B\left(  \xi_{s}\right)  \right)  ds=\xi_{0}\left(  \theta\right)  +\int%
_{0}^{t}\xi_{s}\left(  DB\left(  \xi_{s}\right)  ^{T}\cdot\theta\right)  ds.
\label{var}%
\end{equation}

\end{definition}

The definition on an open interval $[0,T)$ (possibly infinite) is similar. The
aim of this section is to prove the following result.

\begin{theorem}
\label{Thm flow PDE nonlinear}Given $\xi_{0}\in\mathcal{M}$, on a sufficiently
small time interval $\left[  0,T\right]  $, there exists a unique
current-valued solution for the PDE (\ref{PDE nonlinear with B}), defined on a
maximal interval $[0,T_{\xi_{0}}).$ It is given by the unique maximal
current-valued solution of the flow equation (\ref{xi}).
\end{theorem}

The proof consists in proving that $\xi\in C([0,T_{\xi_{0}});\mathcal{M})$ is
a current-valued solution for the PDE (\ref{PDE nonlinear with B}) if and only
if it is a solution of the flow equation (\ref{xi});\ when this is done, all
statements of the theorem are proved, because we already know that equation
(\ref{xi}) has a unique local solution in the space of currents.

In order to prove the previous claim of equivalence between
(\ref{PDE nonlinear with B}) and (\ref{xi}) consider, for given $T>0$ and
$b\in C\left(  \left[  0,T\right]  ;C_{b}^{2}\left(  \mathbb{R}^{d}%
,\mathbb{R}^{d}\right)  \right)  $, the auxiliary \textit{linear} PDE
\begin{equation}
\left\{
\begin{array}
[c]{l}%
\frac{\partial\xi}{\partial t}+\operatorname{div}\left(  b\xi\right)
=(\xi\cdot\nabla)b\\
\xi(0)=\xi_{0}%
\end{array}
\right.  \label{eq pde}%
\end{equation}

The definition of solution is analogous to the nonlinear case, just replacing
$B\left(  \xi\right)  $ by $b$. For this equation we shall prove:

\begin{lemma}
\label{lemma equivalence}A function $\xi\in C([0,T];\mathcal{M})$ is a
current-valued solution for the PDE (\ref{eq pde}) if and only if it is given
by%
\[
\xi_{t}=\varphi_{\sharp}^{t,b}\xi_{0}.
\]

\end{lemma}

The proof of this lemma occupies the next two subsections. When this result is
reached, we can prove Theorem \ref{Thm flow PDE nonlinear} with the following
simple argument: if $\xi\in C([0,T];\mathcal{M})$ is a current-valued solution
for the PDE (\ref{PDE nonlinear with B}), then it is a current-valued solution
for the PDE (\ref{eq pde}) with $b:=B\left(  \xi\right)  $, hence $\xi
_{t}=\varphi_{\sharp}^{t,b}\xi_{0}=\varphi_{\sharp}^{t,B\left(  \xi\right)
}\xi_{0}$, namely $\xi$ solves the flow equation (\ref{xi}). Conversely, if
$\xi$ solves the flow equation (\ref{xi}), namely $\xi_{t}=\varphi_{\sharp
}^{t,B\left(  \xi\right)  }\xi_{0}$, setting $b:=B\left(  \xi\right)  $ we
have that $\xi_{t}=\varphi_{\sharp}^{t,b}\xi_{0}$, hence by the lemma it
solves the PDE (\ref{eq pde}) with $b=B\left(  \xi\right)  $, hence it solves
(\ref{PDE nonlinear with B}).

\subsection{From the flow to the PDE}

In this subsection we prove one half of Lemma \ref{lemma equivalence},
precisely that $\xi_{t}$ defined by $\xi_{t}=\varphi_{\sharp}^{t,b}\xi_{0}$ is
a current-valued solution of the PDE (\ref{eq pde}). Let $\theta$ be a test
function, so that%
\[
\xi_{t}(\theta)=\varphi_{\sharp}^{t,b}\xi_{0}(\theta)=\xi_{0}\left(
\varphi_{\sharp}^{t,b}\theta\right)  =\xi_{0}\left(  D\varphi^{t,b}(\cdot
)^{T}\theta\left(  \varphi^{t,b}(\cdot)\right)  \right)  .
\]
Hence (the time derivative commutes with $\xi_{0}$ since $\xi_{0}$ is linear
continuous)
\begin{align*}
\frac{d}{dt}\xi_{t}(\theta)  &  =\frac{d}{dt}\xi_{0}\left(  D\varphi
^{t,b}(\cdot)^{T}\theta\left(  \varphi^{t,b}(\cdot)\right)  \right) \\
&  =\xi_{0}\left(  \frac{d}{dt}\left(  D\varphi^{t,b}(\cdot)^{T}\right)
\theta\left(  \varphi^{t,b}(\cdot)\right)  \right)  +\xi_{0}\left(
D\varphi^{t,b}(\cdot)^{T}\frac{d}{dt}\left(  \theta\left(  \varphi^{t,b}%
(\cdot)\right)  \right)  \right) \\
&  =\xi_{0}\left(  D\varphi^{t,b}(\cdot)^{T}Db_{t}(\varphi^{t,b}(\cdot
))^{T}\theta\left(  \varphi^{t,b}(\cdot)\right)  \right)  +\xi_{0}\left(
D\varphi^{t,b}(\cdot)^{T}D\theta(\varphi^{t,b}(\cdot))\frac{d}{dt}\left(
\varphi^{t,b}(\cdot)\right)  \right) \\
&  =\xi_{0}\left(  D\varphi^{t,b}(\cdot)^{T}Db_{t}(\varphi^{t,b}(\cdot
))^{T}\theta\left(  \varphi^{t,b}(\cdot)\right)  \right)  +\xi_{0}\left(
D\varphi^{t,b}(\cdot)^{T}D\theta(\varphi^{t,b}(\cdot))b_{t}(\varphi
^{t,b}(\cdot))\right) \\
&  =I_{1}+I_{2}%
\end{align*}
we have%
\[
I_{1}=\xi_{0}\left(  \varphi_{\sharp}^{t,b}\left(  Db^{T}\theta\right)
\right)  =\varphi_{\sharp}^{t,b}\xi_{0}(Db^{T}\theta)=\xi_{t}(Db%
^{T}\theta).
\]
And%
\[
I_{2}=\xi_{0}\left(  \varphi_{\sharp}^{t,b}D\theta b\right)
=\varphi_{\sharp}^{t,b}\xi_{0}\left(  D\theta b\right)  =\xi_{t}\left(
D\theta b\right)  .
\]
Hence
\[
\frac{d}{dt}\xi_{t}(\theta)=\xi_{t}(Db^{T}\theta)+\xi_{t}(D\theta b).
\]
This is equation (\ref{var}) (with $b$ in place of $B\left(  \xi\right)  $),
which completes the proof.

\subsection{From the PDE to the flow}

In this subsection we prove the other half of Lemma \ref{lemma equivalence}:
if $\xi$ is a current-valued solution of the PDE (\ref{eq pde}), then $\xi
_{t}=\varphi_{\sharp}^{t,b}\xi_{0}$.

Since the computation, by means of regularizations and commutator lemma, may
obscure the underlying argument, let us first provide the proof in the case
smooth fields. In such a case, from Proposition
\ref{Prop push forward vector fields} we have (we denote $\varphi^{t,b}$ by
$\varphi_{t}$ for simplicity)
\[
\xi_{t}(x)=D\varphi_{t}(\varphi_{t}^{-1}(x))\xi_{0}(\varphi_{t}^{-1}(x))
\]
which is equivalent to
\begin{equation}
\frac{\partial}{\partial t}\left[  (D\varphi_{t})^{-1}(x)\xi_{t}(\varphi
_{t}(x))\right]  =0. \label{contogiusto}%
\end{equation}
To compute this derivative we will make use of the following rule
\[
d(D\varphi_{t})^{-1}=-(D\varphi_{t})^{-1}Db(\varphi_{t}).
\]
Here and in the following we assume that $D\varphi_{t}$ is a unitary matrix
and $b$ is divergence free regular vector field.

Let us compute the derivative (\ref{contogiusto}),%
\begin{align}
\frac{\partial}{\partial t}\left[  (D\varphi_{t})^{-1}(x)\xi_{t}(\varphi
_{t}(x))\right]   &  =(D\varphi_{t})^{-1}\left[  -Db_{t}(\varphi_{t})\xi
_{t}(\varphi_{t})+\partial_{t}\xi_{t}(\varphi_{t})+D\xi_{t}(\varphi_{t}%
)b_{t}(\varphi_{t})\right] \label{contogiusto2}\\
&  =(D\varphi_{t})^{-1}\left[  -(\xi_{t}\cdot\nabla)b_{t}+\partial_{t}\xi
_{t}+(b_{t}\cdot\nabla)\xi_{t}\right]  (\varphi_{t})\nonumber
\end{align}
and the term in the brackets is equal to zero when $\xi_{t}$ solves equation
(\ref{eq pde}). In the previous computations and also below it is convenient
to keep in mind that, given a vector field $\theta:\mathbb{R}^{d}%
\rightarrow\mathbb{R}^{d}$, its Jacobian matrix is given by
\[
D\theta:=\left(
\begin{array}
[c]{c}%
\nabla\theta_{1}^{T}\\
\nabla\theta_{2}^{T}\\
\nabla\theta_{3}^{T}%
\end{array}
\right)  .
\]

Let us now go back to currents. Given a current-valued solution $\xi$ of the
PDE (\ref{eq pde}), we regularize it as
\[
v_{i}^{\epsilon}(t,x):=(\xi_{t}\ast\theta^{\epsilon}e_{i})(x)=\xi_{t}%
(\theta^{\epsilon}(x-.)e_{i}),\quad\text{for }1\leq i\leq3
\]
Here $\theta^{\epsilon}(x)=\epsilon^{-3}\theta(\epsilon^{-1}x)$ is a
mollifier, and $\{e_{i}\}_{1\leq i\leq3}$ is the canonical base of
$\mathbb{R}^{3}$.

Using equation (\ref{var}) (with $b$ in place of $B\left(  \xi\right)  $), we
see that $v_{\epsilon}$ solves
\[
v_{i}^{\epsilon}(t,x)=v_{i}^{\epsilon}(0,x)+\int_{0}^{t}\left\{  \xi
_{s}(\nabla(\theta^{\epsilon}(x-.))\cdot b_{s}\;e_{i})+[(\xi_{s}\cdot
\nabla)b_{s}](\theta^{\epsilon}(x-.))\right\}  ds
\]
Define now
\begin{equation}
\left(  \mathcal{R}_{1}^{\epsilon}[b_{t},\xi_{t}](x)\right)  _{i}:=\xi
_{t}(\nabla(\theta^{\epsilon}(x-.))\cdot b_{t}\;e_{i})+(b_{t}(x)\cdot
\nabla)\xi_{t}(\theta^{\epsilon}(x-.)e_{i})\quad1\leq i\leq3 \label{Resto1}%
\end{equation}%
\[
\mathcal{R}_{2}^{\epsilon}[\xi_{t},b_{t}](x):=\left(
\begin{array}
[c]{c}%
\lbrack(\xi_{t}\cdot\nabla)b_{t}](\theta^{\epsilon}(x-.)e_{1})\\
\cdots\\
\cdots
\end{array}
\right)  -\left[  \left(
\begin{array}
[c]{c}%
\xi_{t}(\theta^{\epsilon}(x-.)e_{1})\\
\cdot\cdot
\end{array}
\right)  \cdot\nabla\right]  b_{t}(x)
\]
so that
\[
v_{i}^{\epsilon}(t,x)=v^{\epsilon}(0,x)+\int_{0}^{t}\left\{  \mathcal{R}%
_{1}^{\epsilon}[b_{s},\xi_{s}](x)+\mathcal{R}_{2}^{\epsilon}[\xi_{s}%
,b_{s}](x)-(b_{s}\cdot\nabla)v_{i}^{\epsilon}(s,x)+(v_{i}^{\epsilon}%
(s,x)\cdot\nabla)b_{s}\right\}  ds
\]
which means (provided continuity in $t$ of the integrand),
\[
\partial_{t}v_{i}^{\epsilon}(t,x)+(b_{t}\cdot\nabla)v_{i}^{\epsilon
}(t,x)-(v_{i}^{\epsilon}(t,x)\cdot\nabla)b_{t}=\mathcal{R}_{1}^{\epsilon
}[b_{t},\xi_{t}](x)+\mathcal{R}_{2}^{\epsilon}[\xi_{t},b_{t}](x)
\]
Plugging this last equation into equation (\ref{contogiusto2}), we obtain
\begin{equation}
\left[  (D\varphi_{t})^{-1}(x)v_{i}^{\epsilon}(t,\varphi_{t}(x))\right]
=\int_{0}^{t}(D\varphi_{s})^{-1}(x)\left(  \mathcal{R}_{1}^{\epsilon}%
[b_{s},\xi_{s}]+\mathcal{R}_{2}^{\epsilon}[\xi_{s},b_{s}]\right)  (\varphi
_{s}(x))ds \label{Pre_conv}%
\end{equation}
If we want (\ref{contogiusto}) to hold, we must verify that the left-hand side
goes to the left hand side of (\ref{contogiusto}) and the right-hand side goes
to $0$, as $\epsilon\rightarrow\infty$. It suffices to obtain this convergence
weakly, thus we test (\ref{Pre_conv}) against a test function $\rho\in
C_{b}^{1}(\mathbb{R}^{3},\mathbb{R}^{3})$ (we need $\rho$ to be differentiable
because of Lemma \ref{comm_lemma_1} and (\ref{test})).
\[
\int\left[  (D\varphi_{t})^{-1}(x)v_{i}^{\epsilon}(t,\varphi_{t}(x))\right]
\cdot\rho(x)dx=\int\int_{0}^{t}(D\varphi_{s})^{-1}(x)\left(  \mathcal{R}%
_{1}^{\epsilon}[b_{s},\xi_{s}]+\mathcal{R}_{2}^{\epsilon}[\xi_{s}%
,b_{s}]\right)  (\varphi_{s}(x))ds\cdot\rho(x)dx
\]

If we have a closer look at the right-hand side, we see that we need that the
commutator goes to zero when tested against the function
\begin{equation}
\overline{\rho}_{s}(x)=(D\varphi_{s})^{-1}(\varphi_{s}^{-1}(x))\rho
(\varphi_{s}^{-1}(x)) \label{test}%
\end{equation}
If this test function satisfies the assumptions on Lemma \ref{comm_lemma_1}
and \ref{comm_lemma_2}, we can conclude. In particular, we ask that it is
bounded together with his first derivative,
\[
\Vert\overline{\rho}_{s}\Vert_{\infty}\leq\Vert(D\varphi_{s})^{-1}%
\Vert_{\infty}\Vert{\rho_{s}}\Vert_{\infty}%
\]%
\[
\Vert D\overline{\rho}_{s}\Vert_{\infty}\leq\Vert D^{2}\left(  \varphi
_{s}^{-1}\right)  \Vert_{\infty}\Vert\rho\Vert_{\infty}+\Vert(D\varphi
_{s})^{-1}\Vert_{\infty}^{2}\Vert D\rho\Vert_{\infty}%
\]

\begin{lemma}
\label{comm_lemma_1} Let $\rho,b\in C_{b}^{1}(\mathbb{R}^{3},\mathbb{R}^{3})$,
there exists a constant $C$ independent of $\epsilon$ such that,
\[
\left\vert \int\mathcal{R}_{1}^{\epsilon}[b,\xi](x)\cdot\rho(x)dx\right\vert
\leq\epsilon C\left\vert \xi\right\vert _{\mathcal{M}}\Vert D\rho\Vert
_{\infty}\Vert b\Vert_{\infty}%
\]

\end{lemma}

\begin{proof}
By (\ref{Resto1}), we have
\[
\int\mathcal{R}_{1}^{\epsilon}[b,\xi](x)\cdot\rho(x)dx=\sum_{i=1}^{3}%
\int\left[  \xi_{t}(\nabla(\theta^{\epsilon}(x-.))\cdot b\;e_{i}%
)+(b(x)\cdot\nabla)\xi_{t}(\theta^{\epsilon}(x-.)e_{i})\right]  \cdot\rho(x)dx
\]
If we consider $\xi$ to be a $3$-dimensional measure $(d\xi_{1},d\xi_{2}%
,d\xi_{3})$, we obtain%
\begin{align*}
\int\mathcal{R}_{1}^{\epsilon}[b,\xi](x)\cdot\rho(x)dx  &  =\sum_{i=1}%
^{3}\iint\nabla\theta^{\epsilon}(x-y)\cdot(b(x)-b(y))\rho_{i}(x)d\xi
_{i}(y)dx\\
&  =-\sum_{i=1}^{3}\iint\theta^{\epsilon}(b(x)-b(y))\nabla\rho_{i}(x)d\xi
_{i}(y)dx
\end{align*}
If we assume that the current can be swapped with (1) the integration, (2) the
derivative in $x$, and (3) the scalar product by $b(x)$ we can repeat the same
reasoning to obtain
\[
\int\mathcal{R}_{1}^{\epsilon}[b,\xi](x)\cdot\rho(x)dx=-\sum_{i=1}^{3}%
\xi\left(  \int\theta^{\epsilon}(x-y)\left(  b(x)-b(y)\right)  \nabla\rho
_{i}(x)dx\right)
\]
Taking the absolute value on both sides we get
\[
\left\vert \int\mathcal{R}_{1}^{\epsilon}[b,\xi](x)\cdot\rho(x)dx\right\vert
\leq3\left\vert \xi\right\vert _{\mathcal{M}}\Vert D\rho\Vert_{\infty}%
\Vert\nabla b\Vert_{\infty}\sup_{y\in\mathbb{R}^{3}}\left(  \int%
\theta^{\epsilon}(x-y)|x-y|dx\right)
\]
Now, a change of variables in the integral does the trick and we obtain the
desired estimation with the constant equal to
\[
C:=3\int\theta(x)|x|dx
\]

\end{proof}

\begin{lemma}
\label{comm_lemma_2} Let $\rho\in C_{b}(\mathbb{R}^{3},\mathbb{R}^{3})$, $b\in
C_{b}^{2}(\mathbb{R}^{3},\mathbb{R}^{3})$. There exists a constant $C$
independent of $\epsilon$ such that
\[
\left\vert \int\mathcal{R}_{2}^{\epsilon}[\xi,b](x)\cdot\rho(x)dx\right\vert
\leq\epsilon C\left\vert \xi\right\vert _{\mathcal{M}}\Vert\rho\Vert_{\infty
}\Vert D^{2}b\Vert_{\infty}%
\]

\end{lemma}

\begin{proof}
As in the proof of Lemma \ref{comm_lemma_1}, we obtain
\[
\int\mathcal{R}_{2}^{\epsilon}[\xi,b](x)\cdot\rho(x)dx=\sum_{i=1}^{3}%
\xi\left(  \int\theta^{\epsilon}(x-\cdot)(\partial_{i}b(\cdot)-\partial
_{i}b(x))\cdot\rho(x)dx\right)
\]
The proof follows as in Lemma \ref{comm_lemma_1} and the final constant $C$ is
the same.
\end{proof}

\section{Continuous dependence on initial conditions}

Recall that a local time of existence and uniqueness for the flow equation
(\ref{xi}) exists for every $\xi_{0}\in\mathcal{M}$ and its size, in the proof
based on contraction principle, depends only on $\left\vert \xi_{0}\right\vert
_{\mathcal{M}}$. Therefore, if we have a sequence $\xi_{0}^{n}%
\overset{d}{\rightarrow}\xi_{0}$, since weakly convergent sequences are
bounded, there exists a common time interval $\left[  0,T\right]  $ of
existence and uniqueness, with $T>0$.

\begin{theorem}
\label{thm cont dep}Given $\xi_{0},\xi_{0}^{n}\in\mathcal{M}$, $n\in
\mathbb{N}$, with $\lim_{n\rightarrow\infty}\left\Vert \xi_{0}^{n}-\xi
_{0}\right\Vert =0$, let $\left[  0,T\right]  $ be a common time interval of
existence and uniqueness for the flow equation (\ref{xi}) and denote by
$\xi,\xi^{n}\in C\left(  \left[  0,T\right]  ;\mathcal{M}\right)  $ the
corresponding solutions. Then $\xi^{n}\rightarrow\xi$ in $C\left(  \left[
0,T\right]  ;\mathcal{M}_{w}\right)  $ and more precisely there exists a
constant $C>0$ (depending on the $\left\Vert \cdot\right\Vert $-norms of
$\xi_{0},\xi_{0}^{n}$ and on $T$) such that
\[
\sup_{t\in\left[  0,T\right]  }\left\Vert \xi_{t}^{n}-\xi_{t}\right\Vert \leq
C\left\Vert \xi_{0}^{n}-\xi_{0}\right\Vert .
\]

\end{theorem}

\begin{proof}
\textbf{Step 1}. There exists a constant $C_{0}>0$ such that
\begin{equation}
\left\Vert \xi\right\Vert _{T}\leq C_{0},\qquad\left\Vert \xi^{n}\right\Vert
_{T}\leq C_{0} \label{uniform inequalities}%
\end{equation}
uniformly in $n\in\mathbb{N}$. Indeed, the time $T$ can be reached in a finite
number of small steps $T_{R}$ related to the contraction principle, namely the
application of Theorem \ref{Thm existence and uniqueness}. On each small
interval the uniform-in-time $\left\Vert .\right\Vert $-norm is controlled by
the $\left\Vert .\right\Vert $-norm of the initial condition of that time
interval. The inequalities (\ref{uniform inequalities}) readily follow.
Finally, from (\ref{estimate between norms}) it follows also%
\begin{equation}
\sup_{t\in\left[  0,T\right]  }\left\vert \xi_{t}\right\vert _{\mathcal{M}%
}\leq C_{0}^{\prime},\qquad\sup_{t\in\left[  0,T\right]  }\left\vert \xi
_{t}^{n}\right\vert _{\mathcal{M}}\leq C_{0}^{\prime}
\label{uniform inequalities 2}%
\end{equation}
for some $C_{0}^{\prime}>0$.

\textbf{Step 2}.

We have%
\begin{align*}
\left\vert \xi_{t}\left(  \theta\right)  -\xi_{t}^{n}\left(  \theta\right)
\right\vert  &  =\left\vert \varphi_{\sharp}^{t,B(\xi)}\xi_{0}\left(
\theta\right)  -\varphi_{\sharp}^{t,B(\xi^{n})}\xi_{0}^{n}\left(
\theta\right)  \right\vert \\
&  \leq\left\vert \varphi_{\sharp}^{t,B(\xi)}\xi_{0}\left(  \theta\right)
-\varphi_{\sharp}^{t,B(\xi)}\xi_{0}^{n}\left(  \theta\right)  \right\vert
+\left\vert \varphi_{\sharp}^{t,B(\xi)}\xi_{0}^{n}\left(  \theta\right)
-\varphi_{\sharp}^{t,B(\xi^{n})}\xi_{0}^{n}\left(  \theta\right)  \right\vert
\\
&  =\left\vert \left(  \xi_{0}-\xi_{0}^{n}\right)  \left(  \varphi_{\sharp
}^{t,B(\xi)}\theta\right)  \right\vert +\left\vert \xi_{0}^{n}\left(
\varphi_{\sharp}^{t,B(\xi)}\theta-\varphi_{\sharp}^{t,B(\xi^{n})}%
\theta\right)  \right\vert \\
&  \leq\left\Vert \xi_{0}-\xi_{0}^{n}\right\Vert \left(  \Vert\varphi_{\sharp
}^{t,B(\xi)}\theta\Vert_{\infty}+\text{Lip}(\varphi_{\sharp}^{t,B(\xi)}%
\theta)\right)  +\left\vert \xi_{0}^{n}\right\vert _{\mathcal{M}}\Vert
\varphi_{\sharp}^{t,B(\xi)}\theta-\varphi_{\sharp}^{t,B(\xi^{n})}\theta
\Vert_{\infty}.
\end{align*}
Now, from (\ref{uniform inequalities}), (\ref{Dphi}) and (\ref{Dphi_Lipschitz}%
) there exist $C_{11},C_{12}>0$ such that
\[
\Vert\varphi_{\sharp}^{t,B(\xi)}\theta\Vert_{\infty}=\Vert D\varphi^{t,B(\xi
)}(\cdot)^{T}\theta(\varphi^{t,B(\xi)}(\cdot))\Vert_{\infty}\leq\Vert
D\varphi^{t,B(\xi)}\Vert_{\infty}\Vert\theta\Vert_{\infty}\leq C_{11}%
\]%
\[
\text{Lip}(\varphi_{\sharp}^{t,B(\xi)}\theta)\leq\Vert D^{2}\varphi^{t,B(\xi
)}(\cdot)\Vert_{\infty}\Vert\theta\Vert_{\infty}+\Vert D\varphi^{t,B(\xi
)}\Vert_{\infty}\text{Lip}(\theta)\leq C_{12}.
\]
Moreover,%
\begin{align*}
\Vert\varphi_{\sharp}^{t,B(\xi)}\theta-\varphi_{\sharp}^{t,B(\xi^{n})}%
\theta\Vert_{\infty}  &  =\Vert D\varphi^{t,B(\xi)}(\cdot)^{T}\theta
(\varphi^{t,B(\xi)}(\cdot))-D\varphi^{t,B(\xi^{n})}(\cdot)^{T}\theta
(\varphi^{t,B(\xi^{n})}(\cdot))\Vert_{\infty}\\
&  \leq\Vert D\varphi^{t,B(\xi)}(\cdot)^{T}\theta(\varphi^{t,B(\xi)}%
(\cdot))-D\varphi^{t,B(\xi)}(\cdot)^{T}\theta(\varphi^{t,B(\xi^{n})}%
(\cdot))\Vert_{\infty}\\
&  +\Vert D\varphi^{t,B(\xi)}(\cdot)^{T}\theta(\varphi^{t,B(\xi^{n})}%
(\cdot))-D\varphi^{t,B(\xi^{n})}(\cdot)^{T}\theta(\varphi^{t,B(\xi^{n})}%
(\cdot))\Vert_{\infty}%
\end{align*}%
\begin{align*}
&  \leq\Vert D\varphi^{t,B(\xi)}\Vert_{\infty}\text{Lip}(\theta)\Vert
\varphi^{t,B(\xi)}-\varphi^{t,B(\xi^{n})}\Vert_{\infty}+\Vert D\varphi
^{t,B(\xi)}-D\varphi^{t,B(\xi^{n})}\Vert_{\infty}\Vert\theta\Vert_{\infty}\\
&  \leq\Vert D\varphi^{t,B(\xi)}\Vert_{\infty}\Vert\varphi^{t,B(\xi)}%
-\varphi^{t,B(\xi n)}\Vert_{\infty}+\Vert D\varphi^{t,B(\xi)}-D\varphi
^{t,B(\xi^{n})}\Vert_{\infty}\\
&  \leq te^{C_{B}T(\left\Vert \xi\right\Vert _{T}+1)}C_{B}e^{C_{B}T(\left\Vert
\xi\right\Vert _{T}+1)}\Vert\xi-\xi^{n}\Vert_{t}\\
&  +tC_{B}e^{C_{B}\left(  T+1\right)  (\left\Vert \xi\right\Vert
_{T}+\left\Vert \xi^{n}\right\Vert _{T}+2)}\left(  1+C_{B}T(\left\Vert
\xi\right\Vert _{T}+1)e^{C_{B}T(\left\Vert \xi\right\Vert _{T}+1)}\right)
\Vert\xi-\xi^{n}\Vert_{t}.
\end{align*}
Thus there exists $C_{13}>0$ such that%
\[
\Vert\varphi_{\sharp}^{t,B(\xi)}\theta-\varphi_{\sharp}^{t,B(\xi^{n})}%
\theta\Vert_{\infty}\leq tC_{13}\Vert\xi-\xi^{n}\Vert_{t}.
\]
Collecting these bounds, for every $T_{0}\in\left[  0,T\right]  $ we get
\[
\left\Vert \xi-\xi^{n}\right\Vert _{T_{0}}\leq\left\Vert \xi_{0}-\xi_{0}%
^{n}\right\Vert \left(  C_{11}+C_{12}\right)  +T_{0}C_{13}\left\vert \xi
_{0}^{n}\right\vert _{\mathcal{M}}\Vert\xi-\xi^{n}\Vert_{T_{0}}.
\]
This proves the theorem if $T_{0}$ is small enough, say $T_{0}\leq\frac
{1}{2C_{13}C_{0}^{\prime}}$;\ but the constant $2C_{13}C_{0}^{\prime}$ does
not vary when we repeat the argument on the interval $\left[  T_{0}%
,2T_{0}\right]  $ and so on (until we cover $\left[  0,T\right]  $) and thus
in a finite number of steps we get the result on $\left[  0,T\right]  $.
\end{proof}

\section{Interacting filaments and their mean field limit}

\subsection{Interacting filaments as current dynamics}

Let $K\in\mathcal{U}C_{b}^{3}\left(  \mathbb{R}^{d},\mathbb{R}^{d\times
d}\right)  $ be given. An example of $K$ which heuristically motivates the
investigation of filaments done here can be, in $d=3$, a smooth approximation
of Biot-Savart kernel. Consider a set of $N$ curves in $\mathbb{R}^{d}$,
$\gamma_{t}^{i,N}\left(  \sigma\right)  $, parametrized by $\sigma\in\left[
0,1\right]  $, time dependent, which satisfy the equations%

\begin{align}
\frac{\partial}{\partial t}\gamma_{t}^{i,N}\left(  \sigma\right)   &
=\sum_{j=1}^{N}\alpha_{j}^{N}\int_{0}^{1}K\left(  \gamma_{t}^{i,N}\left(
\sigma\right)  -\gamma_{t}^{j,N}\left(  \sigma^{\prime}\right)  \right)
\frac{\partial}{\partial\sigma^{\prime}}\gamma_{t}^{j,N}\left(  \sigma
^{\prime}\right)  d\sigma^{\prime}\label{curve eq}\\
\gamma_{0}^{i,N}\text{ given, }i  &  =1,...,N\nonumber
\end{align}
for some sequence of weights $\left\{  \alpha_{j}^{N}\right\}  $. This
dynamics of curves can be reformulated as a dynamics of currents of the form
\begin{equation}
\xi_{t}=\varphi_{\sharp}^{t,K\ast\xi}\xi_{0} \label{flow of currents}%
\end{equation}
for a suitable $\xi_{0}$. Let us explain this reformulation.

To the family of curves we associate the current $\xi_{t}:C_{b}\left(
\mathbb{R}^{d};\mathbb{R}^{d}\right)  \rightarrow\mathbb{R}$ defined as%

\begin{equation}
\xi_{t}\left(  \theta\right)  :=\sum_{j=1}^{N}\alpha_{j}^{N}\int_{0}^{1}%
\theta\left(  \gamma_{t}^{j,N}\left(  \sigma\right)  \right)  \frac{\partial
}{\partial\sigma}\gamma_{t}^{j,N}\left(  \sigma\right)  d\sigma
\label{def curr}%
\end{equation}
or more formally, in the vein of empirical measures,%

\begin{equation}
\xi_{t}=\sum_{j=1}^{N}\alpha_{j}^{N}\int_{0}^{1}\delta_{\gamma_{t}%
^{j,N}\left(  \sigma\right)  }\frac{\partial}{\partial\sigma}\gamma_{t}%
^{j,N}\left(  \sigma\right)  d\sigma. \label{current of a set of curves}%
\end{equation}

\begin{theorem}
\label{thm curves as currents}If $\left\{  \gamma_{t}^{i,N}\left(
\sigma\right)  ,i=1,...,N,t\in\left[  0,T\right]  ,\sigma\in\left[
0,1\right]  \right\}  $ is a family of $C^{1}\left(  \left[  0,T\right]
\times\left[  0,1\right]  ;\mathbb{R}^{d}\right)  $ functions which satisfies
the identities (\ref{curve eq}), then the current $\xi_{t}:C_{b}\left(
\mathbb{R}^{d};\mathbb{R}^{d}\right)  \rightarrow\mathbb{R}$, $t\in\left[
0,T\right]  $, defined by (\ref{def curr}) satisfies equation
(\ref{flow of currents}). Conversely, if $\xi\in C\left(  \left[  0,T\right]
;\mathcal{M}_{w}\right)  $ is the unique current-valued solution of equation
(\ref{flow of currents}) with $\xi_{0}$ defined by (\ref{def curr}) (for
$t=0$) with respect to a given family of $C^{1}$ initial curves $\left\{
\gamma_{0}^{i,N}\left(  \sigma\right)  ,i=1,...,N,\sigma\in\left[  0,1\right]
\right\}  $, then the representation (\ref{def curr}) holds where $\gamma
_{t}^{i,N}\left(  \sigma\right)  $ is defined as%
\begin{equation}
\gamma_{t}^{i,N}\left(  \sigma\right)  =\varphi^{t.K\ast\xi}\left(  \gamma
_{0}^{i,N}\left(  \sigma\right)  \right)  \label{def gamma t}%
\end{equation}
and the curves $\gamma_{t}^{i,N}\left(  \sigma\right)  $ satisfy the
identities (\ref{curve eq}).
\end{theorem}

\begin{proof}
Let us prove the first direction: from the general definition of push-forward,
we get%
\begin{align*}
\left(  \varphi_{\sharp}^{t,K\ast\xi}\xi_{0}\right)  \left(  \theta\right)
&  =\xi_{0}\left(  \varphi_{\sharp}^{t,K\ast\xi}\theta\right)  =\xi_{0}\left(
\left(  D\varphi^{t,K\ast\xi}\right)  ^{T}\theta\circ\varphi^{t,K\ast\xi
}\right) \\
&  =\sum_{j=1}^{N}\alpha_{j}^{N}\int_{0}^{1}\widetilde{\theta}\left(
\gamma_{0}^{j,N}\left(  \sigma\right)  \right)  \frac{\partial}{\partial
\sigma}\gamma_{0}^{j,N}\left(  \sigma\right)  d\sigma
\end{align*}
where%
\[
\widetilde{\theta}=\left(  D\varphi^{t,K\ast\xi}\right)  ^{T}\theta
\circ\varphi^{t,K\ast\xi}%
\]
and we want to prove that this is equal to
\[
\xi_{t}\left(  \theta\right)  =\sum_{j=1}^{N}\alpha_{j}^{N}\int_{0}^{1}%
\theta\left(  \gamma_{t}^{j,N}\left(  \sigma\right)  \right)  \frac{\partial
}{\partial\sigma}\gamma_{t}^{j,N}\left(  \sigma\right)  d\sigma.
\]
Thus it is sufficient to prove%
\[
\int_{0}^{1}\theta\left(  \gamma_{t}^{j,N}\left(  \sigma\right)  \right)
\frac{\partial}{\partial\sigma}\gamma_{t}^{j,N}\left(  \sigma\right)
d\sigma=\int_{0}^{1}\widetilde{\theta}\left(  \gamma_{0}^{j,N}\left(
\sigma\right)  \right)  \frac{\partial}{\partial\sigma}\gamma_{0}^{j,N}\left(
\sigma\right)  d\sigma.
\]
To prove this, notice that
\[
\left(  K\ast\xi_{t}\right)  \left(  x\right)  =\sum_{j=1}^{N}\alpha_{j}%
^{N}\int_{0}^{1}K\left(  x-\gamma_{t}^{j,N}\left(  \sigma\right)  \right)
\frac{\partial}{\partial\sigma}\gamma_{t}^{j,N}\left(  \sigma\right)
d\sigma.
\]
Therefore the equations (\ref{curve eq}) for the interaction of curves can be
rewritten as%
\[
\frac{\partial}{\partial t}\gamma_{t}^{i,N}\left(  \sigma\right)  =\left(
K\ast\xi_{t}\right)  \left(  \gamma_{t}^{i,N}\left(  \sigma\right)  \right)
.
\]
This means that
\[
\gamma_{t}^{i,N}\left(  \sigma\right)  =\varphi^{t.K\ast\xi}\left(  \gamma
_{0}^{i,N}\left(  \sigma\right)  \right)  .
\]
Now, from this fact, we can deduce the identity above. Indeed,
\begin{align*}
\int_{0}^{1}\theta\left(  \gamma_{t}^{j,N}\left(  \sigma\right)  \right)
\frac{\partial}{\partial\sigma}\gamma_{t}^{j,N}\left(  \sigma\right)  d\sigma
&  =\int_{0}^{1}\theta\left(  \varphi^{t.K\ast\xi}\left(  \gamma_{0}%
^{i,N}\left(  \sigma\right)  \right)  \right)  \frac{\partial}{\partial\sigma
}\varphi^{t.K\ast\xi}\left(  \gamma_{0}^{i,N}\left(  \sigma\right)  \right)
d\sigma\\
&  =\int_{0}^{1}\left(  D\varphi^{t.K\ast\xi}\right)  ^{T}\left(  \gamma
_{0}^{i,N}\left(  \sigma\right)  \right)  \left(  \theta\circ\varphi
^{t.K\ast\xi}\right)  \left(  \gamma_{0}^{i,N}\left(  \sigma\right)  \right)
\frac{\partial}{\partial\sigma}\gamma_{0}^{i,N}\left(  \sigma\right)  d\sigma.
\end{align*}

Let us now prove the opposite direction. Let us assume that $\xi_{t}$
satisfies equation (\ref{flow of currents}) with $\xi_{0}$ defined by
(\ref{def curr}) for $t=0$, with respect to a given family of $C^{1}$ initial
curves $\left\{  \gamma_{0}^{i,N}\left(  \sigma\right)  ,i=1,...,N,\sigma
\in\left[  0,1\right]  \right\}  $. Then we have that%
\begin{align*}
\xi_{t}(\theta)  &  =\left(  \varphi^{t.K\ast\xi}\xi_{0}\right)  \left(
\theta\right)  =\xi_{0}\left(  \varphi_{\sharp}^{t,K\ast\xi}\theta\right)
=\xi_{0}\left(  \left(  D\varphi^{t.K\ast\xi}\right)  ^{T}\theta\circ
\varphi^{t.K\ast\xi}\right) \\
&  =\sum_{j=1}^{N}\alpha_{j}^{N}\int_{0}^{1}\left(  D\varphi^{t.K\ast\xi
}\right)  ^{T}\left(  \gamma_{0}^{j,N}(\sigma)\right)  \theta\left(
\varphi^{t.K\ast\xi}\left(  \gamma_{0}^{j,N}(\sigma)\right)  \right)
\frac{\partial}{\partial\sigma}\gamma_{0}^{j,N}\left(  \sigma\right)
d\sigma\\
&  =\sum_{j=1}^{N}\alpha_{j}^{N}\int_{0}^{1}\theta\left(  \varphi^{t.K\ast\xi
}\left(  \gamma_{0}^{j,N}(\sigma)\right)  \right)  \frac{\partial}%
{\partial\sigma}\left(  \varphi^{t.K\ast\xi}\left(  \gamma_{0}^{j,N}%
(\sigma)\right)  \right)  d\sigma
\end{align*}
Let us define $\gamma_{t}^{j,N}(\sigma)$ by (\ref{def gamma t}). Then the
representation (\ref{def curr}) holds. Moreover, from (\ref{def gamma t}) we
have, for each $\sigma$,
\[
\frac{\partial}{\partial t}\gamma_{t}^{i,N}\left(  \sigma\right)  =\left(
K\ast\xi_{t}\right)  \left(  \gamma_{t}^{i,N}\left(  \sigma\right)  \right)
\]
which is precisely (\ref{curve eq}) due to the already established form of
$\xi_{t}$.
\end{proof}

The reformulation above provides first of all an existence and uniqueness result:

\begin{corollary}
Assume $K\in\mathcal{U}C_{b}^{3}\left(  \mathbb{R}^{d},\mathbb{R}^{d\times
d}\right)  $. For every $N\in\mathbb{N}$ and every family of $C^{1}$ curves
$\left\{  \gamma_{0}^{i,N}\left(  \sigma\right)  ,i=1,...,N,\sigma\in\left[
0,1\right]  \right\}  $, there exists a unique maximal solution \\
$\left\{\gamma_{t}^{i,N}\left(  \sigma\right)  ,i=1,...,N,t\in\lbrack0,T_{f}%
),\sigma\in\left[  0,1\right]  \right\}  $ of equations (\ref{curve eq}) in
the class of \\
$C^{1}\left(  [0,T_{f})\times\left[  0,1\right]  ;\mathbb{R}%
^{d}\right)  ^{N}$ functions.
\end{corollary}

\subsection{Mean field result\label{subsection mean field}}

The next theorem proves two important results: first, if a family of initial
curves approximates a current at time $t=0$, then the solutions of the
filament equations converge to the solution of the vector valued PDE. The
second, related result is that each curve of the family becomes, in the limit
$N\rightarrow\infty$, closer and closer to the solutions $\overline{\gamma
}_{t}^{i,N}\left(  \sigma\right)  $ of equation (\ref{filament and mean field}%
), precisely%
\begin{align}
\frac{\partial}{\partial t}\overline{\gamma}_{t}^{i,N}\left(  \sigma\right)
&  =\left(  K\ast\xi_{t}\right)  \left(  \overline{\gamma}_{t}^{i,N}\left(
\sigma\right)  \right) \label{filament and mean field 2}\\
\overline{\gamma}_{0}^{i,N}  &  =\gamma_{0}^{i,N}.\nonumber
\end{align}
This equation describes the interaction of a filament with the mean field
$\xi_{t}$. This is the core of the concept of mean field theory. Notice that,
without further assumptions on the convergence of $\gamma_{0}^{i,N}$ (that we
do not assume, since a typical example is the case of random independent
initial conditions), it is not true that $\overline{\gamma}_{t}^{i,N}$
converges. The second part of the next theorem only claims that $\gamma
_{t}^{i,N}$\ and $\overline{\gamma}_{t}^{i,N}$ are close to each other.

As a technical remark, notice that if $\xi_{t}$ exists on a time interval
$\left[  0,T\right]  $, then $K\ast\xi_{t}$ satisfies the regularity
conditions of Lemma \ref{lemma on K} and therefore there exists a unique
time-dependent $C^{1}$-curve $\overline{\gamma}_{t}^{i,N}$ (for each $i,N$),
solution of equation (\ref{filament and mean field 2}).

\begin{theorem}
Let, for every $N\in\mathbb{N}$, $\left\{  \gamma_{0}^{i,N}\left(
\sigma\right)  ,i=1,...,N,\sigma\in\left[  0,1\right]  \right\}  $ be a family
of $C^{1}$ curves. Assume that the associated currents at time zero%
\begin{equation}
\xi_{0}^{N}=\sum_{j=1}^{N}\alpha_{j}^{N}\int_{0}^{1}\delta_{\gamma_{0}%
^{j,N}\left(  \sigma\right)  }\frac{\partial}{\partial\sigma}\gamma_{0}%
^{j,N}\left(  \sigma\right)  d\sigma\label{initial current}%
\end{equation}
converge weakly to a current $\xi_{0}\in\mathcal{M}$. Let $T>0$ be any time
such that on $\left[  0,T\right]  $ there are unique current-valued solutions
$\xi_{t}^{N}$ and $\xi_{t}$ to equation (\ref{flow of currents}) with respect
to the initial conditions $\xi_{0}^{N}$ or $\xi_{0}$; notice that such a time
exists because the initial currents $\xi_{0}^{N}$ and $\xi_{0}$ are
equibounded (since they converge weakly); moreover, notice that $\xi_{t}^{N}$
has the form%
\[
\xi_{t}^{N}=\sum_{j=1}^{N}\alpha_{j}^{N}\int_{0}^{1}\delta_{\gamma_{t}%
^{j,N}\left(  \sigma\right)  }\frac{\partial}{\partial\sigma}\gamma_{t}%
^{j,N}\left(  \sigma\right)  d\sigma
\]
corresponding to curve-solutions to equation (\ref{curve eq}) and that
$\xi_{t}$ is the unique solution to the vector-valued PDE (\ref{eq pde}). Let
$\overline{\gamma}_{t}^{i,N}$ be the solution to the mean field equation
(\ref{filament and mean field 2}).

Then:

i) the currents $\xi^{N}$ converge in $C\left(  \left[  0,T\right]
;\mathcal{M}_{w}\right)  $ to the current $\xi$.

ii) $\lim_{N\rightarrow\infty}\sup_{\left(  t,\sigma\right)  \in\left[
0,T\right]  \times\left[  0,1\right]  }\left\vert \gamma_{t}^{i,N}\left(
\sigma\right)  -\overline{\gamma}_{t}^{i,N}\left(  \sigma\right)  \right\vert
=0$.
\end{theorem}

\begin{proof}
Part (i) is a straightforward consequence of Theorem \ref{thm cont dep} on the
continuous dependence on initial conditions.

As to part (ii), denoting as above by $\xi_{t}^{N},\overline{\xi}_{t}^{N}$
respectively the currents associated to the families $\gamma_{t}%
^{i,N},\overline{\gamma}_{t}^{i,N}$ (see (\ref{def curr})), we have
\begin{align*}
\left\vert \gamma_{t}^{i,N}\left(  \sigma\right)  -\overline{\gamma}_{t}%
^{i,N}\left(  \sigma\right)  \right\vert  &  \leq\int_{0}^{t}\left\vert
(K\ast\xi_{s}^{N})(\gamma_{s}^{i,N}\left(  \sigma\right)  )-(K\ast\xi
_{s})(\overline{\gamma}_{s}^{i,N}\left(  \sigma\right)  )\right\vert ds\\
&  \leq\int_{0}^{t}\left\vert (K\ast\xi_{s}^{N})(\gamma_{s}^{i,N}\left(
\sigma\right)  )-(K\ast\xi_{s}^{N})(\overline{\gamma}_{s}^{i,N}\left(
\sigma\right)  )\right\vert ds\\
&  +\int_{0}^{t}\left\vert (K\ast\xi_{s}^{N})(\overline{\gamma}_{s}%
^{i,N}\left(  \sigma\right)  )-(K\ast\xi_{s})(\overline{\gamma}_{s}%
^{i,N}\left(  \sigma\right)  )\right\vert ds.
\end{align*}
From two of the properties of "$K\ast\xi$" proved in Lemma \ref{lemma on K},
we have (taking also the supremum in $\sigma\in\left[  0,1\right]  $)%
\begin{align*}
\left\vert \gamma_{t}^{i,N}\left(  \sigma\right)  -\overline{\gamma}_{t}%
^{i,N}\left(  \sigma\right)  \right\vert  &  \leq C\int_{0}^{t}\left\Vert
DK\ast\xi_{s}^{N}\right\Vert _{\infty}\left\vert \gamma_{s}^{i,N}\left(
\sigma\right)  -\overline{\gamma}_{s}^{i,N}\left(  \sigma\right)  \right\vert
ds+C\int_{0}^{t}\left\Vert \xi_{s}^{N}-\xi_{s}\right\Vert ds\\
&  \leq C\int_{0}^{t}\left\vert \gamma_{s}^{i,N}\left(  \sigma\right)
-\overline{\gamma}_{s}^{i,N}\left(  \sigma\right)  \right\vert ds+C\int%
_{0}^{t}\left\Vert \xi_{s}^{N}-\xi_{s}\right\Vert ds
\end{align*}
where we have denoted a generic constant by $C$ and we have used that
$\sup_{t\in\left[  0,T\right]  }\left\Vert \xi_{s}^{N}\right\Vert <\infty$, as
we know from the first part of the proof (e.g. since they converge in
$C\left(  \left[  0,T\right]  ;\mathcal{M}_{w}\right)  $). We also know, from
the first part, that $\overline{\xi}^{N}\rightarrow\xi$ in in $C\left(
\left[  0,T\right]  ;\mathcal{M}_{w}\right)  $. Then it is sufficient to apply
Gronwall's Lemma to obtain the claim of part (ii).
\end{proof}

Sometimes one has a probabilistic framework of the following kind.\ We have a
filtered probability space $(\Omega,(\mathcal{F}_{t})_{t\geq0},\mathcal{F}%
,\mathbb{P})$ and the separable Banach space $\mathcal{C}=C\left(  \left[
0,1\right]  ;\mathbb{R}^{d}\right)  $ with the Borel $\sigma$-algebra
$\mathcal{B}\left(  \mathcal{C}\right)  $; we call \textit{random curve} in
$\mathbb{R}^{d}$ any measurable map from $(\Omega,\mathcal{F},\mathbb{P})$ to
$\left(  \mathcal{C},\mathcal{B}\left(  \mathcal{C}\right)  \right)  $. We use
the notation $\gamma$ also for a random curve. The image measure $\mu$ of such
map is the \textit{law of the random curve}. It is a probability measure on
$\left(  \mathcal{C},\mathcal{B}\left(  \mathcal{C}\right)  \right)  $.

Then consider, for every $N\in\mathbb{N}$, a family $\left\{  \gamma_{0}%
^{i,N},i=1,...,N\right\}  $ of random curves and consider the associated
currents $\xi_{0}^{N}$ defined as in (\ref{initial current}), which now are
random currents, namely measurable mappings from $\left(  \Omega,F,P\right)  $
to the space $\mathcal{M}$ endowed with its Borel $\sigma$-algebra
$\mathcal{B}\left(  \mathcal{M}\right)  $. Let $\xi_{0}$ be a random current.
For all $\omega\in\Omega$, solve uniquely the flow equation
(\ref{flow of currents}) with the initial conditions $\xi_{0}^{N}\left(
\omega\right)  $ and $\xi_{0}\left(  \omega\right)  $ and call $\xi_{t}%
^{N}\left(  \omega\right)  $ and $\xi_{t}\left(  \omega\right)  $ the
corresponding solutions. Assume that all the the whole family of currents
$\xi_{0}^{N}\left(  \omega\right)  $, $\xi_{0}\left(  \omega\right)  $ when
$N$ and $\omega$ vary, are equibounded. Then take some $T>0$ such that unique
solutions $\xi_{t}^{N}\left(  \omega\right)  $ and $\xi_{t}\left(
\omega\right)  $ exist. For every $t\in\left[  0,T\right]  $, $\xi_{t}^{N}$
and $\xi_{t}$ are random currents (namely they are measurable), by the
continuous dependence on initial conditions, Theorem \ref{thm cont dep}.
Assume that $\xi_{0}^{N}$ converges in probability to $\xi_{0}$. Then it is
easy to show that, for every $t\in\left[  0,T\right]  $, $\xi_{t}^{N}$
converges in probability to $\xi_{t}$, and also that $\left\Vert \xi_{\cdot
}^{N}-\xi_{\cdot}\right\Vert _{T}$ converges in probability to zero.

\subsection{Propagation of chaos}

In this section we assume that the vorticity is the same for each vortex, namely $\alpha_{j}^{N}=\frac{1}{N}$ for every $j\leq N$, to ensure that independence is maintained as $N \to \infty$.

To every curve $\gamma\in C^{1}([0,1],\mathbb{R}^{d})$, we can associate a
current , which will also called $\gamma$ with a slight abuse of notation, in
this way
\[
\gamma(\theta):=\int_{0}^{1}\theta(\gamma(\sigma))\frac{\partial}%
{\partial\sigma}\gamma(\sigma)d\sigma
\]
for $\theta\in C_{b}\left(  \mathbb{R}^{d},\mathbb{R}^{d}\right)  $. The
definition of the tensor product $\gamma\otimes\gamma^{\prime}$ is%
\[
\left(  \gamma\otimes\gamma^{\prime}\right)  \left(  \theta,\theta^{\prime
}\right)  =\gamma\left(  \theta\right)  \gamma^{\prime}\left(  \theta^{\prime
}\right)  .
\]

We fix a filtered probability space $(\Omega,(\mathcal{F}_{t})_{t\geq
0},\mathcal{F},\mathbb{P})$ and, following the notion given in the previous
subsection, we consider random curves. We say that a family $(\gamma
_{i})_{1\leq i\leq N}$ of random curves is symmetric or exchangeable if its
law is independent of permutations of the indexes. We start with the following
general result, for random currents independent of time.

\begin{theorem}
\label{propagation_chaos} Let $\xi$ be a current and, for every $N\in
\mathbb{N}$, let $\gamma^{N}:=(\gamma^{i,N})_{1\leq i\leq N}$ be a symmetric
family of random-$C^{1}([0,1],\mathbb{R}^{d})$ curves. We call $\xi^{N}$ the
empirical measure associated with the family $\gamma^{N}$. Suppose that, for
every $\theta\in C_{b}(\mathbb{R}^{3},\mathbb{R}^{3})$,
\begin{equation}
|\xi|_{\mathcal{M}}<\infty,\qquad|\gamma^{i,N}(\theta)|\leq C\quad
\mathbb{P}-a.s.\label{assu chaos}%
\end{equation}
uniformly in $i,N$ and that
\begin{equation}
\lim_{N\rightarrow\infty}\mathbb{E}\left[  |\xi^{N}(\theta)-\xi(\theta
)|\right]  =0.\label{conv_weak_L1}%
\end{equation}
Then, for every fixed $r\in\mathbb{N}$ and for every family of test functions
$(\theta_{1},\cdots,\theta_{r})\in C_{b}(\mathbb{R}^{d},\mathbb{R}^{d})^{r}$,
it holds%
\[
\lim_{N\rightarrow\infty}\mathbb{E}\left[  \left(  \gamma^{1,N}\otimes
\cdots\otimes\gamma^{ir,N}\right)  (\theta_{1},\cdots,\theta_{r})\right]
=\prod_{i=1}^{r}\xi(\theta_{i}).
\]

\end{theorem}

\begin{proof}
Without loss of generality, we prove the theorem in the case $k=2$. For every
$\theta_{1},\theta_{2}$ bounded Lipschitz continuous functions in
$\mathbb{R}^{d}$ we have
\begin{align}
|\mathbb{E}\left[  \gamma^{1,N}(\theta_{1})\gamma^{2,N}(\theta_{2})\right]
-\xi(\theta_{1})\xi(\theta_{2})|\leq &  |\mathbb{E}\left[  \gamma^{1,N}%
(\theta_{1})\gamma^{2,N}(\theta_{2})\right]  -\mathbb{E}\left[  \xi^{N}%
(\theta_{1})\xi^{N}(\theta_{2})\right]  |\label{term filaments}\\
&  +|\mathbb{E}\left[  \xi^{N}(\theta_{1})\xi^{N}(\theta_{2})\right]
-\xi(\theta_{1})\xi(\theta_{2})| \label{term currents}%
\end{align}
The second term, \eqref{term currents}, goes to zero because of
(\ref{assu chaos})-\eqref{conv_weak_L1}:
\begin{align*}
&  \left\vert \mathbb{E}\left[  \xi^{N}(\theta_{1})\xi^{N}(\theta_{2})\right]
-\xi(\theta_{1})\xi(\theta_{2})\right\vert \\
&  \leq\left\vert \mathbb{E}\left[  \left(  \xi^{N}(\theta_{1})-\xi(\theta
_{1})\right)  \xi^{N}(\theta_{2})\right]  \right\vert +\left\vert
\mathbb{E}\left[  \left(  \xi^{N}(\theta_{2})-\xi(\theta_{2})\right)
\xi(\theta_{1})\right]  \right\vert \\
&  \leq\mathbb{E}\left[ \vert\xi^{N}(\theta_{1}) - \xi(\theta_{1})
\vert\right]  \mathbb{E}\left[  \vert\xi^{N}(\theta_{2}) \vert\right]
+\left\vert \xi\right\vert _{\mathcal{M}} \mathbb{E}\left[  \left\vert \xi
^{N}(\theta_{2}) -\xi(\theta_{2})\right\vert \right]
\end{align*}

To study \eqref{term filaments} we use the symmetry of $\gamma$,%
\begin{align*}
|\mathbb{E}\left[  \gamma_{1}(\theta_{1})\gamma_{2}(\theta_{2})\right]
-\mathbb{E}\left[  \xi^{N}(\theta_{1})\xi^{N}(\theta_{2})\right]  |= &
|\mathbb{E}\left[  \gamma_{1}(\theta_{1})\gamma_{2}(\theta_{2})\right]  \\
&  -\frac{N^{2}-N}{N^{2}}\mathbb{E}\left[  \gamma_{1}(\theta_{1})\gamma
_{2}(\theta_{2})\right]  -\frac{1}{N}\mathbb{E}\left[  \gamma_{1}(\theta
_{1})^{2}\right]  |\\
= &  \frac{1}{N}\left(  \mathbb{E}\left[  |\gamma_{1}(\theta_{1})\gamma
_{2}(\theta_{2})|\right]  +\mathbb{E}\left[  \gamma_{1}(\theta_{1}%
)^{2}\right]  \right)
\end{align*}
The expectations are bounded because of \eqref{assu chaos}, hence the last
term goes to zero.
\end{proof}

Now we want to apply the previous theorem to our filaments. We verify in the
following Lemma that the dynamic of filaments satisfies Theorem
\ref{propagation_chaos}, under suitable assumptions on the initial condition.

\begin{lemma}
\label{lemma chaos}Given a family $\gamma_{0}:=\{\gamma_{0}^{i,N}\}_{1\leq
i\leq N}$ of random variables on $C^{1}(\mathbb{R}^{3},\mathbb{R}^{3})$, and a
current $\xi_{0}$, we assume

\begin{enumerate}
\item $(\gamma_{0}^{1, N}, \dots, \gamma_{0}^{N, N})$ are exchangeable.

\item $\left\vert \gamma^{i,N}_{0}(\theta)\right\vert \leq C$, for a.e.
$\omega$, uniformly in $i$ and $N$.

\item $\vert\xi_{0}\vert_{\mathcal{M}} < \infty$

\item $\lim_{N\rightarrow\infty}\mathbb{E}\left[  \Vert\xi_{0}^{N}-\xi
_{0}\Vert\right]  =0 $
\end{enumerate}

There exists a time $T>0$ such that the solutions $\gamma_{t}:=\{\gamma
_{t}^{i,N}\}_{1\leq i\leq N}$ and $\xi_{t}$ of equations \eqref{curve eq} and
\eqref{flow of currents} starting respectively from $\gamma_{0}$ and $\xi_{0}$
satisfy conditions 1) - 4) at every time $t\in\left[  0,T\right]  $.
\end{lemma}

\begin{proof}
Exchangeability is clearly preserved by the system of filaments, because there
is no other randomness and the dynamics of each filament is perfectly equal to
the one of the others.

To prove that $\gamma_{t}^{i, N}(\theta)$ is bounded, we use
\eqref{def gamma t} and its derivative and we obtain
\begin{align*}
\gamma_{t}^{i, N}(\theta) =  & \int_{0}^{1} \theta(\gamma_{t}^{i,N}%
)\frac{\partial}{\partial\sigma}\gamma_{t}^{i,N} d \sigma\\
=  &  \int_{0}^{1} \theta(\varphi^{K*\xi_{t}}(\gamma_{0}^{i, N}))(K*\xi
_{t})(\gamma_{0}^{i, N})\frac{\partial}{\partial\sigma}\gamma_{0}^{i,N}
d\sigma= \gamma_{0}^{i, N} (\bar{\theta})
\end{align*}
where $\bar{\theta}(x) := \theta(\varphi^{K*\xi_{t}}(x))(K*\xi_{t})(x)$ is
bounded continuous because of Lemma \ref{lemma on K}.

The third property follows immediately from Theorem
\ref{Thm existence and uniqueness}.

The last property, 4), is a direct consequence of Theorem \ref{thm cont dep}.
\end{proof}

\begin{corollary}
Under the assumptions of Lemma \ref{lemma chaos}, for every fixed
$r\in\mathbb{N}$, every family of test functions $(\theta_{1},\cdots
,\theta_{r})\in C_{b}(\mathbb{R}^{d},\mathbb{R}^{d})^{r}$ and every
$t\in\left[  0,T\right]  $, it holds%
\[
\lim_{N\rightarrow\infty}\mathbb{E}\left[  \left(  \gamma_{t}^{1,N}%
\otimes\cdots\otimes\gamma_{t}^{ir,N}\right)  (\theta_{1},\cdots,\theta
_{r})\right]  =\prod_{i=1}^{r}\xi_{t}(\theta_{i}).
\]

\end{corollary}
\vspace{1cm}

\noindent
{\bf Acknowledgements}: Hakima Bessaih's research is partially supported by the NSF grant DMS-1418838.

\end{document}